\documentclass[11pt,leqno]{article}
\usepackage{amsmath, amscd, amsthm, amssymb, graphics, xypic, mathrsfs, setspace, fancyhdr, times, bm, pdfsync, enumitem}
\usepackage[usenames, dvipsnames, svgnames, table]{xcolor}
\usepackage[letterpaper,top=1.05in, bottom=1.05in, left=1.05in, right=1.05in]{geometry}
\usepackage[colorlinks=true,pagebackref=true]{hyperref} 
\hypersetup{backref}

\newcommand{\tensor}{\otimes}
\newcommand{\colim}{\operatorname{colim}}

\newcommand{\Spec}{\operatorname{Spec}}

\newcommand{\iso}{\cong}

\newcommand{\weq}{\simeq}
\newcommand{\isomto}{{\stackrel{\sim}{\;\longrightarrow\;}}}

\newcommand{\sma}{{\scriptstyle{\wedge}\,}}

\newcommand{\GW}{\mathbf{GW}}
\newcommand{\W}{\mathbf{W}}
\newcommand{\Hom}{\operatorname{Hom}}

\renewcommand{\hom}{\operatorname{Hom}}

\newcommand{\real}{{\mathbb R}}
\newcommand{\RR}{\real}
\newcommand{\PP}{\mathbb{P}}

\newcommand{\cplx}{{\mathbb C}}
\newcommand{\Q}{{\mathbb Q}}
\newcommand{\Z}{{\mathbb Z}}

\newcommand{\aone}{{\mathbb A}^1}
\newcommand{\pone}{{\mathbb P}^1}

\newcommand{\gm}[1]{{{\mathbb G}_{m}^{#1}}}

\newcommand{\MW}{\mathrm{MW}}

\newcommand{\et}{\text{\'et}}
\newcommand{\Et}{\operatorname{{\acute E}t}}

\newcommand{\hop}[1]{\mathscr{H}_{\bullet}({#1})}

\newcommand{\bpi}{\bm{\pi}}
\newcommand{\piaone}{{\bpi}^{\aone}}
\newcommand{\bpia}{\piaone}

\newcommand{\Nis}{{\operatorname{Nis}}}
\newcommand{\Zar}{\operatorname{Zar}} 

\newcommand{\Sm}{\mathrm{Sm}}

\newcommand{\Spc}{\mathrm{Spc}}

\newcommand{\K}{{{\mathbf K}}}

\newcommand{\KM}{\K^{\mathrm{M}}}

\newcommand{\Loc}{\mathrm{L}}
\newcommand{\Laone}{\Loc_{\aone}}
\newcommand{\Singaone}{\operatorname{Sing}^{\aone}\!\!}

\newcommand{\ZZ}{\Z}

\newcommand{\Addresses}{{
 \bigskip
 \footnotesize

 A.~Asok, Department of Mathematics, University of Southern California, 3620 S. Vermont Ave.,
 Los Angeles, CA 90089-2532, United States; \textit{E-mail address:} \url{asok@usc.edu}

 \medskip

 J.~Fasel, Institut Fourier - UMR 5582, Universit\'e Grenoble Alpes, CS 40700, 38058 Grenoble Cedex 9; France \textit{E-mail address:} \url{jean.fasel@gmail.com}

 \medskip

 T.B.\ Williams, Department of Mathematics, The University of British Columbia, 1984 Mathematics Road
Vancouver, B.C. Canada V6T 1Z2; \textit{E-mail address:} \url{tbjw@math.ubc.ca}

}}

\newcounter{intro}
\theoremstyle{plain}
\newtheorem{thm}{Theorem}[subsection]

\newtheorem{lem}[thm]{Lemma}
\newtheorem{cor}[thm]{Corollary}
\newtheorem{prop}[thm]{Proposition}
\newtheorem*{claim*}{Claim} 

\newtheorem*{thm*}{Theorem}
\newtheorem*{problem*}{Problem}

\newtheorem{thmintro}{Theorem}

\newtheorem{conjintro}[thmintro]{Conjecture}

\theoremstyle{definition}
\newtheorem{defn}[thm]{Definition}

\theoremstyle{remark}
\newtheorem{rem}[thm]{Remark}
\newtheorem{remintro}[thmintro]{Remark}

\numberwithin{equation}{subsection}

\begin{document}
\pagestyle{fancy}
\renewcommand{\sectionmark}[1]{\markright{\thesection\ #1}}
\fancyhead{}
\fancyhead[LO,R]{\bfseries\footnotesize\thepage}
\fancyhead[LE]{\bfseries\footnotesize\rightmark}
\fancyhead[RO]{\bfseries\footnotesize\rightmark}
\chead[]{}
\cfoot[]{}
\setlength{\headheight}{1cm}

\author{Aravind Asok\thanks{Aravind Asok was partially supported by National Science Foundation Award DMS-1254892.}\and Jean Fasel \and Ben Williams}

\title{{\bf Motivic spheres and the image of the Suslin--Hurewicz map}}
\date{}
\maketitle

\begin{abstract}
We show that an old conjecture of A.A. Suslin characterizing the image of a Hurewicz map from Quillen K-theory in degree $n$ to Milnor K-theory in degree $n$ admits an interpretation in terms of unstable $\aone$-homotopy sheaves of the general linear group. Using this identification, we establish Suslin's conjecture in degree $5$ for infinite fields having characteristic unequal to $2$ or $3$. We do this by linking the relevant unstable $\aone$-homotopy sheaf of the general linear group to the stable $\aone$-homotopy of motivic spheres.
\end{abstract}

\begin{footnotesize}
\tableofcontents
\end{footnotesize}

\section{Introduction}
The goal of this paper is to explore how concrete computations in $\aone$-homotopy theory of ostensibly geometric origin
have bearing on torsion phenomena in algebraic K-theory. More precisely, we investigate an old conjecture of Suslin, the
formulation of which we now recall. Suppose $F$ is an infinite field. By definition of the plus construction, for any integer $n \geq 1$, the Hurewicz map for the stable general linear group induces a morphism
\[
K^Q_n(F) := \pi_n(BGL(F)^+) \longrightarrow H_n(BGL(F)^+) \isomto H_n(BGL(F));
\]
here and henceforth, homology is taken with integer coefficients, which are suppressed from the notation. Suslin's stabilization theorem \cite[Theorem 3.4]{Suslin} asserts:
\begin{itemize}[noitemsep,topsep=1pt]
\item[(i)] the stabilization maps $BGL_n(F) \to BGL_{n+1}(F)$ induce isomorphisms, functorially in $F$, of the form $H_i(BGL_n(F)) \longrightarrow H_i(BGL_{n+1}(F))$ whenever $i \leq n$; in particular, there is an induced isomorphism $H_n(BGL_n(F)) \to H_n(BGL(F))$;
\item[(ii)] the cokernel of the stabilization map $H_{n}(BGL_{n-1}(F)) \to H_n(BGL_n(F))$ coincides, functorially in $F$ with $K^M_n(F)$,
\end{itemize}
the Milnor K-theory of the field $F$. Putting all these facts together, one obtains a morphism
\[
K^Q_n(F) \longrightarrow H_n(BGL(F)) \longrightarrow H_n(BGL_n(F)) \longrightarrow K^M_n(F),
\]
which is functorial in $F$. We will refer to this composite map as the {\em Suslin--Hurewicz map}.

There is a natural graded ring homomorphism $K^M_*(F) \to K^Q_*(F)$ induced by the identification $K^M_1(F) = K^Q_1(F)$ and product maps in K-theory. Suslin showed that the composite of the natural homomorphism and the Suslin--Hurewicz homomorphism
\[
K^M_n(F) \longrightarrow K^Q_n(F) \longrightarrow K^M_n(F)
\]
coincides with multiplication by $(-1)^n(n-1)!$. It follows that the image of the Suslin--Hurewicz map $K^Q_n(F) \to K^M_n(F)$ contains $(n-1)!K^M_n(F)$. Suslin went on to make the following conjecture.

\begin{conjintro}[{see \cite[p. 370, after Corollary 4.4]{Suslin}}]
For any infinite field $F$, the image of the Suslin--Hurewicz map $K^Q_n(F) \to K^M_n(F)$ coincides with $(n-1)! K^M_n(F)$.
\end{conjintro}

Suslin refers to this conjecture as ``very delicate". As evidence for this assessment, he analyzed the first
interesting case of this conjecture, i.e., $n = 3$. In that case, he showed the conjecture was equivalent to the degree
$3$ case of the Milnor conjecture on quadratic forms, i.e., that the map $K^M_3(F)/2 \to I^3(F)/I^4(F)$ is an isomorphism \cite[Proposition 4.5]{Suslin}.

Suslin's conjecture was soon shown to hold in degree $3$. Indeed, Merkurjev and Suslin established it by careful
analysis of indecomposable $K_3$ of a field (see \cite[Proof of Proposition 11.10]{MerkurjevSuslin}). Independently,
Rost established the Milnor conjecture on quadratic forms in degree $3$, building on work of Arason \cite{Arason}; his
work remains unpublished \cite{Rost}, but see \cite{JacobRost} for more detailed references and some history. In
conjunction with the equivalence that Suslin proved (mentioned in the previous paragraph), Rost's work presents another affirmation of Suslin's conjecture in degree $3$. Nowadays, there are many proofs of the Milnor conjecture on quadratic forms in all degrees, beginning with the work of Orlov--Vishik--Voevodsky \cite{OVV} (see also \cite{Morel05,RondigsOstvaerMC}).

\begin{remintro}
Merkurjev explained that Suslin also analyzed the $n = 4$ case of the conjecture, but never published anything. Various results in K-theory imply a positive answer to Suslin's question for certain classes of fields (e.g., algebraically closed fields) but, excepting the degree $3$ case, as far as we are aware, there are no general results about Suslin's conjecture.
\end{remintro}

The construction of the Suslin--Hurewicz map was generalized to local rings with infinite residue fields by Nesterenko--Suslin \cite[\S 4]{NesterenkoSuslin} and Guin \cite[\S 4]{Guin}. Moreover, the image of this more general map has the same properties as discussed above. While Suslin initially stated his conjecture only for infinite fields, in light of the subsequent generalizations of the Suslin--Hurewicz homomorphism, it seems reasonable to replace the infinite field $F$ by a local ring with infinite residue field. In support of this generalization of Suslin's conjecture, we offer the following result.

\begin{thmintro}[See Theorem~\ref{thm:suslinsconjecturedegree5}]
\label{thmintro:suslinsconjecture}
If $k$ is an infinite field having characteristic unequal to $2$ or $3$, then Suslin's conjecture holds in degree $5$ for any essentially smooth local $k$-algebra $A$, i.e., the Suslin--Hurewicz map $K^Q_5(A) \to K^M_5(A)$ has image precisely $24 K^M_5(A)$.
\end{thmintro}

To establish Theorem~\ref{thmintro:suslinsconjecture}, we compare the Suslin--Hurewicz homomorphism with another map $K^Q_n(F) \to K^M_n(F)$ that naturally appeared in computations of $\aone$-homotopy sheaves of the general linear group \cite{AsokFaselSpheres}. The main goal of Section \ref{s:suslinshurewiczmap} is to establish this comparison, which is achieved in Theorem~\ref{thm:suslinmorphismscoincide}. The results rely on the techniques of \cite{AHW,AHWII}, building on ideas of F. Morel \cite{MField}, and recent work of Schlichting \cite{Schlichting} which allow us to interpret the $\aone$-homotopy computations in terms of homology of certain discrete simplicial groups (in brief, we analyze the effect of forcing the homology of the discrete general linear group to be $\aone$-invariant). Along the way, we establish some results about group homology that might be of independent interest, e.g., Theorem~\ref{thm:maintechnical}.

Granted the results mentioned in the preceding paragraph, Suslin's conjecture for general $n$ may be interpreted as a statement about the structure of $\aone$-homotopy sheaves of $BGL_n$. Once reformulated in terms of $\aone$-homotopy sheaves, we establish Suslin's conjecture in degree $5$ by showing that the relevant $\aone$-homotopy sheaf computation may be related to a computation of an unstable $\aone$-homotopy sheaf of a motivic sphere, refining a key computation of \cite{AWW}. In particular, we establish the following result, which is one of the main results of Section \ref{s:pi4refinement}.

\begin{thmintro}[See Theorem~\ref{thm:maincomputation}]
\label{thmintro:maincomputation}
If $k$ is a field that has characteristic unequal to $2$ or $3$, then there is a short exact sequence of the form
\[
0 \longrightarrow \K^{\mathrm{M}}_{5}/24 \longrightarrow \bpi_4^{\aone}({\pone}^{\sma 3}) \longrightarrow \mathbf{GW}^3_4 \longrightarrow 0.
\]
\end{thmintro}

The key idea that permits this refinement is a comparison of unstable and stable computations of $\aone$-homotopy sheaves. The proof of Theorem~\ref{thmintro:maincomputation} relies on the beautiful computation of the first stable $\aone$-homotopy sheaf of the motivic sphere spectrum by R{\"ondigs}--Spitzweck--{\O}stv{\ae}r \cite{RSO}. After the discussion of Section~\ref{s:suslinshurewiczmap}, Suslin's conjecture can be viewed as a statement about a measure of deviation from stability in the $\aone$-homotopy sheaves of the general linear group (here, stabilization refers to the map from the general linear group to the stable general linear group). The results of Section~\ref{s:pi4refinement} then support the idea that the measure of deviation from stability is itself already stable in the sense of stable $\aone$-homotopy theory. In particular, we hope the technique of proof can be adapted to shed light on Suslin's conjecture for other values of $n$.

\subsubsection*{Acknowledgements}
The first author would like to thank Sasha Merkurjev for explaining Suslin's approach to his eponymous conjecture in degree $4$; even though this makes no appearance here, it was still an essential input. The authors would also like to thank Marc Levine and the University of Duisburg-Essen where the initial idea of this paper was conceived and Kirsten Wickelgren for her collaboration in an early stage of this project. Finally, the authors would like to thank Oliver R\"ondigs for helpful discussions about \cite{RSO}, and for comments and corrections on a draft of this work.

\subsubsection*{Preliminaries/Notation}
Throughout the paper, $k$ will denote a fixed base field. We write $\Sm_k$ for the category of schemes that are separated, smooth and have finite type over $\Spec k$. We write $\Spc_k$ for the category of simplicial presheaves on $\Spec k$; objects of this category will typically be written using a script font (e.g., $\mathscr{X}$). Our notation in Section \ref{s:suslinshurewiczmap} follows \cite{AHW,AHWII}; we summarize most of what we will need from these papers in Section~\ref{ss:snandhomologyofsing}. For example, we will write $\mathrm{R}_{\Zar}$ for the Zariski fibrant replacement functor with respect to the injective Zariski local model structure on $\Spc_k$, and $\mathrm{R}_{\Nis}$ for the corresponding construction in the Nisnevich local model structure. Our notation for $\aone$-homotopy sheaves follows that of \cite{AWW} on which this paper builds.

\section{The Suslin--Hurewicz homomorphism revisited}
\label{s:suslinshurewiczmap}
The goal of this section is to compare two homomorphisms from Quillen K-theory to Milnor K-theory: the first is the Suslin--Hurewicz homomorphism described in the introduction, and the second, constructed in \cite{AsokFaselSpheres}, arises naturally in motivic homotopy theory (it is related to a homomorphism defined by Suslin using Mennicke symbols). Using some ideas from $\aone$-homotopy theory and some recent results of M. Schlichting, we will demonstrate that the two homomorphisms coincide.

In order to compare the two homomorphisms, we study homological stabilization results for spaces constructed out of general linear groups; Theorem \ref{thm:maintechnical} is related to the constructions of \cite{HutchinsonWendt} who consider special linear groups, though our proof is somewhat different (in particular, it uses results of \cite{AHW,AHWII} in place of corresponding results from \cite{MField}). The main result of this section is Theorem \ref{thm:suslinmorphismscoincide}, which essentially shows that Suslin's conjecture from the introduction, may be reformulated as providing a precise description of an $\aone$-homotopy sheaf, building on the ideas of \cite{AsokFaselSpheres}.

\subsection{The sheaf $\mathbf{S}_n$ via homology of simplicial groups}
\label{ss:snandhomologyofsing}
In \cite{AsokFaselSpheres}, the first and second authors studied the $\aone$-homotopy theory of $BGL_n$. Mirroring the situation for unitary groups in classical algebraic topology, there is a range in which these homotopy sheaves are ``stable" in the sense that they agree with the $\aone$-homotopy sheaves of the stable general linear group $BGL$, which may be described in terms of algebraic K-theory. The first homotopy sheaf of $BGL_n$ lying outside of this stable range is that in degree $n$: this sheaf is an extension of an algebraic K-theory sheaf by a ``non-stable contribution". For $n \geq 3$, the non-stable contribution depends on the parity of $n$ and is phrased in terms of a sheaf we called $\mathbf{S}_{n+1}$ and, if $n$ is even, an additional factor, $\mathbf{I}^{n+1}$. Ultimately, we will interpret Suslin's conjecture from the introduction as a statement about the structure of $\mathbf{S}_{n+1}$. Before doing this, we recall the construction and properties of $\mathbf{S}_{n+1}$ in detail. Using the results of \cite{AHW,AHWII}, we provide a ``concrete" interpretation of the sections of this sheaf over (suitable) local rings in terms of homology of simplicial groups; this approach builds on ideas of F. Morel.

\subsubsection*{Fiber sequences and homotopy sheaves}
Suppose $k$ is a field, and $GL_n$ is the general linear $k$-group scheme. Consider the morphism of schemes $GL_{n-1} \to GL_n$ sending an invertible $(n-1) \times (n-1)$-matrix $M$ to the block matrix $\operatorname{diag}(M,1)$. This morphism induces a map of simplicial classifying spaces $BGL_{n-1} \to BGL_n$ (thought of as simplicial presheaves on $\Sm_k$) that we will refer to as the stabilization map.

For every integer $n \geq 1$ there is an $\aone$-fiber sequence of the form
\[
{\mathbb A}^n \setminus 0 \longrightarrow BGL_{n-1} \longrightarrow BGL_n
\]
(we will refine and explain this fact in Proposition \ref{prop:fibersequences}). Morel showed that ${\mathbb A}^n \setminus 0$ is $\aone$-$(n-2)$-connected and that $\bpi_{n-1}^{\aone}({\mathbb A}^n \setminus 0) \cong \K^{\MW}_n$ where $\K^{\MW}_n$ is Morel's unramified Milnor--Witt K-theory sheaf \cite[Corollary 6.39]{MField}.

By representability of algebraic K-theory in the $\aone$-homotopy category, one may show that $\bpi_i^{\aone}(BGL_n) \cong \K^Q_i$ for $1 \leq i \leq n-1$, where $\K^Q_i$ is the sheafification of the Quillen K-theory presheaf on $\Sm_k$ for the Nisnevich topology. Stringing the associated long exact sequences in $\aone$-homotopy sheaves together for different values of $n$, one obtains a composite morphism
\[
\K^{\MW}_{n+1} = \bpi_{n}^{\aone}({\mathbb A}^{n+1} \setminus 0) \longrightarrow \bpi_n^{\aone}(BGL_n) \longrightarrow \bpi_{n-1}^{\aone}({\mathbb A}^{n}\setminus 0) = \K^{\MW}_n.
\]
This composite map $\K^{\MW}_{n+1} \to \K^{\MW}_n$ is multiplication by $\eta$ if $n$ is even and $0$ if $n$ is odd by
\cite[Lemma 3.5]{AsokFaselSpheres}. Using the quotient map $\K^{\MW}_n \to \K^{\mathrm{M}}_n$ (the latter is an
unramified Milnor K-theory sheaf), which corresponds to forming the quotient by the subsheaf of multiples of $\eta$, one then obtains a commutative diagram of the form
\[
\xymatrix{
\K^{\MW}_{n+1} \ar[d] & & \\
\bpi_n^{\aone}(BGL_n) \ar[d]\ar[r] & \K^{\MW}_{n} \ar[r]\ar[d] & \bpi_{n-1}^{\aone}(BGL_{n-1}) \\
\bpi_{n}^{\aone}(BGL_{n+1})\ar@{..>}[r]^-{\psi_n} & \KM_n;&
}
\]
where the dotted morphism exists for any $n \geq 2$ (\cite[Lemma 3.5, Diagram 3.2]{AsokFaselSpheres}). Since $\bpi_n^{\aone}(BGL_{n+1}) = \K^Q_n$, we conclude that
\[
\psi_n: \K^Q_n \longrightarrow \KM_n,
\]
and we repeat \cite[Definition 3.6]{AsokFaselSpheres}.

\begin{defn}
\label{defn:sn}
For any $n \geq 2$, define $\mathbf{S_n} := \operatorname{coker}(\psi_n)$.
\end{defn}


The next result summarizes key properties of the sheaf $\mathbf{S}_n$.

\begin{prop}
\label{prop:propertiesofsn}
Suppose $n \geq 2$ is an integer.
\begin{enumerate}[noitemsep,topsep=1pt]
 \item There is a canonical morphism $\mu_n: \KM_n \to \K^Q_n$ extending the map induced by the isomorphism $\KM_1 = \K^Q_1$ and the product maps in K-theory.
 \item The composite map $\psi_n \circ \mu_n$ is multiplication by $(n-1)!$.
 \item The canonical epimorphism $\KM_n \to \KM_n/2$ factors through an epimorphism $\mathbf{S}_n \to \KM_n/2$.
 \item The epimorphism $\KM_n \to \mathbf{S}_n$ factors through an epimorphism $\KM_n/(n-1)! \to \mathbf{S}_n$.
\end{enumerate}
\end{prop}

\begin{proof}
The first statement is \cite[Lemma 3.7]{AsokFaselSpheres}, and the latter two statements follow from \cite[Corollary 3.11]{AsokFaselSpheres} and its proof.
\end{proof}

\subsubsection*{Homotopy of the singular construction}
We now appeal to the results of \cite{AHW,AHWII} to recast the above results in terms of homotopy of classifying spaces of simplicial groups. Suppose now that $\mathscr{X}$ is a simplicial presheaf on $\Sm_k$. Write $\Delta^{\bullet}_k$ for the cosimplicial affine $k$-simplex, i.e., the cosimplicial object defined by $n \mapsto \Spec k[x_0,\ldots,x_n]/\langle \sum_i x_i - 1\rangle$ equipped with the usual coface and codegeneracy maps (see, e.g., \cite[p. 88]{MV}). In that case, one defines the singular construction on $\mathscr{X}$ as the diagonal of a bisimplicial object:
\[
\Singaone \mathscr{X} := \operatorname{diag}(\underline{\hom}(\Delta^{\bullet}_k,\mathscr{X})),
\]
where $\underline{\hom}$ is the internal hom in the category of simplicial presheaves.

A list of properties of the singular construction is provided on \cite[p. 87]{MV}. The map
$\mathscr{X} \to \Singaone \mathscr{X}$ is a monomorphism and an $\aone$-weak equivalence, and $\Singaone$ commutes with
the formation of finite limits (in particular, finite products). In the situations of interest to us, the simplicial
presheaf $\Singaone \mathscr{X}$ is already $\aone$-local; the next result summarizes the facts we will need.

\begin{prop}
\label{prop:fibersequences}
Suppose $n \geq 1$ is an integer and $k$ is a field.
\begin{enumerate}[noitemsep,topsep=1pt]
\item There is a Nisnevich local fiber sequence of the form
\[
\Singaone ({\mathbb A}^n \setminus 0) \longrightarrow \Singaone BGL_{n-1} \longrightarrow \Singaone BGL_n.
\]
\item For any smooth affine $k$-scheme $U$, the maps $\Singaone ({\mathbb A}^n \setminus 0)(U) \longrightarrow \mathrm{R}_{\Zar} \Singaone ({\mathbb A}^n \setminus 0)(U)$ and $\Singaone BGL_n(U) \longrightarrow \mathrm{R}_{\Zar} \Singaone BGL_n(U)$ are weak equivalences.
\item The spaces $\mathrm{R}_{\Zar} \Singaone ({\mathbb A}^n \setminus 0)$ and $\mathrm{R}_{\Zar} \Singaone BGL_n$ are
 both Nisnevich local and $\aone$-invariant.
\end{enumerate}
\end{prop}

\begin{proof}
Observe that there is a simplicial fiber sequence
\[
GL_n/GL_{n-1} \longrightarrow BGL_{n-1} \longrightarrow BGL_n
\]
essentially by definition (see \cite[\S 2.3]{AHWII} and Lemma 2.3.1). Applying $\Singaone$ to each term here, in light of \cite[Theorem 5.2.1]{AHW}, we conclude that there is a simplicial fiber sequence of the form
\[
\Singaone GL_n/GL_{n-1} \longrightarrow \Singaone BGL_{n-1} \longrightarrow \Singaone BGL_n
\]
by appeal to \cite[Proposition 2.1.1]{AHWII}. The ``projection onto the first column" map $GL_n/GL_{n-1} \to {\mathbb A}^n \setminus 0$ is a Zariski locally trivial morphism with affine space fibers and thus by \cite[Lemma 4.2.4]{AHWII} the induced map $\Singaone GL_n/GL_{n-1} \to \Singaone ({\mathbb A}^n \setminus 0)$ is a weak equivalence after evaluation on affine schemes and the first result follows. The second and third statements are then contained in \cite[Theorem 5.1.3]{AHW} and \cite[Theorem 2.3.2]{AHWII}.
\end{proof}


\begin{cor}
\label{cor:homotopyofsingularsets}
Suppose $k$ is a field, and $A$ is the local ring of a smooth $k$-scheme $X$ at a point. The following statements hold:
\begin{enumerate}[noitemsep,topsep=1pt]
\item $\pi_i(\Singaone ({\mathbb A}^n \setminus 0)(A)) = \begin{cases} 0 & \text{ if } 1 \leq i \leq n-2. \\ \K^{\MW}_n(A) & \text{ if } i = n-1.\end{cases}$
\item $\pi_i(\Singaone BGL_n(A)) = \K^Q_i(A)$ if $1 \leq i \leq n-1$.
\end{enumerate}
\end{cor}

\begin{proof}
We know $\mathrm{R}_{\Zar} \Singaone ({\mathbb A}^n \setminus 0)$ is Nisnevich local and $\aone$-invariant. In particular, $\bpi_i^{\aone}({\mathbb A}^n \setminus 0) = a_{\Nis}\pi_i(\mathrm{R}_{\Zar} \Singaone ({\mathbb A}^n \setminus 0))$. By \cite[Chapter 6]{MField}, we know that for any $i > 0$, the map $a_{\Zar}\pi_i(\mathrm{R}_{\Zar} \Singaone ({\mathbb A}^n \setminus 0)) \to a_{\Nis}\pi_i(\mathrm{R}_{\Zar} \Singaone ({\mathbb A}^n \setminus 0))$ is an isomorphism, i.e., the Zariski sheafification of the presheaf of homotopy groups is already a Nisnevich sheaf. Thus, for $A$ as in the statement, we see that
\[
\bpi_i^{\aone}({\mathbb A}^n \setminus 0)(A) = \pi_i(\mathrm{R}_{\Zar} \Singaone ({\mathbb A}^n \setminus 0))(A) = \pi_i(\mathrm{R}_{\Zar} \Singaone ({\mathbb A}^n \setminus 0)(A)),
\]
where the last equality follows essentially from the definition of $R_{\Zar}$, i.e., from the fact that $\mathrm{R}_{\Zar} \Singaone ({\mathbb A}^n \setminus 0)$ has Zariski stalks that are fibrant simplicial sets.

On the other hand, for any smooth affine $k$-scheme $U$, the map $\Singaone ({\mathbb A}^n \setminus 0)(U) \to \mathrm{R}_{\Zar} \Singaone ({\mathbb A}^n \setminus 0)(U)$ is a weak equivalence, it follows that the same holds for $U = \Spec A$. The result then follows from \cite[Corollary 6.39]{MField} as this computes the sheaf $\bpi_i^{\aone}({\mathbb A}^n \setminus 0)$ in the relevant cases.

The second statement is deduced in a similar fashion. Using the connectivity statement for $\Singaone ({\mathbb A}^{n+i} \setminus 0)(S)$ mentioned above, the stabilization map $\mathrm{R}_{\Zar} \Singaone BGL_n \to \mathrm{R}_{\Zar} \Singaone BGL$ is an $(n-1)$-equivalence upon evaluation at sections for $A$ as in the statement. The latter space represents algebraic K-theory by \cite[\S 4 Theorem 3.13]{MField} (though see \cite[Theorem 4 and Remark 2 p. 1162]{SchlichtingTripathi} for some mild corrections and to establish the statement in the generality we need).
\end{proof}

\begin{rem}
It follows immediately from Proposition \ref{prop:fibersequences}(ii) and the argument in the beginning of Corollary
\ref{cor:homotopyofsingularsets} that the Zariski sheafification of $U \mapsto \pi_i(\Singaone BGL_n(U))$ is already a
Nisnevich sheaf. For example, the Zariski sheaves $\K^Q_n$ and $\K^{\mathrm{M}}_n$ described in the previous section are
already Nisnevich sheaves. In addition, it follows that if $A$ is the local ring of a smooth $k$-scheme at a point, then
$\bpi_n^{\aone}(BGL_n)(\Spec A)$ coincides with $\pi_n(\Singaone BGL_n(A))$; we will use this identification freely in
the sequel. This provides the first link between $\mathbf{S}_{n+1}$ and the homotopy of $\Singaone BGL_n(A)$.
\end{rem}

\subsection{Homological stability and Milnor K-theory}
\label{ss:homologicalstabilization}
Suslin's conjecture is formulated in terms of homology of the classifying spaces of the discrete groups $GL_n(F)$ and results about homological stabilization for these groups. In the previous section, we saw that the homotopy of certain simplicial groups appeared naturally. Building on the homotopical results of the previous section, we proceed to analyze relative homology of the map $\Singaone BGL_{n} \to \Singaone BGL_{n+1}$. These results are natural simplicial counterparts of the results of Nesterenko--Suslin \cite{NesterenkoSuslin}, which we quickly review. In particular, we establish Lemma \ref{lem:homologystabilizationsingularconstruction}, which is a preliminary homological stabilization result, and Lemma \ref{lem:relativeHurewiczinvariantscomputation} which yields an analog of Suslin's morphism $H_n(BGL_n(F)) \to K^M_n(F)$ in the context of homology of $\Singaone BGL_n$.

\subsubsection*{Review of some results of Nesterenko--Suslin}
We now recall some results of Suslin as extended by Nesterenko--Suslin/Guin. The maps $GL_m \times GL_n \to GL_{m+n}$ given by block sum:
\[
(X_1,X_2) \longmapsto \begin{pmatrix}X_1 & 0 \\ 0 & X_2 \end{pmatrix}
\]
induce maps of classifying spaces $BGL_m \times BGL_n \cong B(GL_n \times GL_m) \to BGL_{m+n}$. For any commutative unital ring $A$, these maps induce external product maps $H_m(BGL_m(A)) \tensor H_n(BGL_n(A)) \to H_{n+m}(BGL_{n+m}(A))$ which are studied in \cite[\S ]{Suslin} and \cite[\S 3]{NesterenkoSuslin}. These external products map equip $\bigoplus_{n \geq 0} H_n(BGL_n(A))$ with the structure of a ring.

A result of Suslin \cite[Corollary 2.7.2]{Suslin}, generalized by Nesterenko--Suslin/Guin, shows that for any local ring $A$ with infinite residue field, the exterior product map
\[
H_1(BGL_1(A))^{\times n} \longrightarrow H_n(BGL_n(A))
\]
factors through a map
\[
\theta: \K^{\mathrm{M}}_n(A) \longrightarrow H_n(BGL_n(A))/H_n(BGL_{n-1}(A)).
\]
Suslin constructed an explicit splitting of this map \cite[Corollaries 2.4.1 and 2.7.4]{Suslin} and concluded in \cite[Theorem 3.4]{Suslin} that $\theta$ is an isomorphism. This result was generalized in \cite[Theorem 3.25]{NesterenkoSuslin}.

\begin{defn}
\label{defn:suslinsmorphism}
The map $s_n$, defined as the composite
\[
H_n(BGL_n(A)) \longrightarrow H_n(BGL_n(A))/H_n(BGL_{n-1}(A)) \stackrel{\theta^{-1}}{\longrightarrow} \K^{\mathrm{M}}_n(A),
\]
will be called {\em Suslin's morphism}.
\end{defn}

The maps $s_n$ just described can be put together to yield:
\[
\bigoplus_{n \geq 0} s_n: \bigoplus_{n \geq 0} H_n(BGL_n(A)) \longrightarrow \bigoplus_{n \geq 0} \K^{\mathrm{M}}_n(A).
\]
By appeal to \cite[Lemma 2.6.1]{Suslin} and \cite[Lemma 3.23]{NesterenkoSuslin}, the direct sum on the left hand side
has naturally the structure of a graded ring. In fact, this ring is graded commutative (essentially because the direct
sum operation on vector spaces is symmetric monoidal). It is well known that the Milnor K-theory ring is graded commutative as well. By \cite[Corollary 3.28]{NesterenkoSuslin}, $\bigoplus_{n \geq 0} s_n$ is in fact a homomorphism of graded rings.

\subsubsection*{Weak homological stability}
We now analyze an analog of Suslin's homomorphism after applying the singular construction.

\begin{lem}
\label{lem:homologystabilizationsingularconstruction}
Suppose $n \geq 1$ is an integer, $k$ is a field and $A$ is an essentially smooth local $k$-algebra. The morphism:
\[
H_i(\Singaone BGL_{n}(A)) \longrightarrow H_i(\Singaone BGL_{n+1}(A))
\]
is an isomorphism for $i \leq n-1$ and split surjective for $i = n$.
\end{lem}

\begin{proof}
For $n$ as in the statement, the fact that the map in question is an isomorphism for $i \leq n-1$ and surjective for $i = n$ follows by combining the relative Hurewicz theorem and the connectivity estimate for $\Singaone ({\mathbb A}^{n+1} \setminus 0) (A)$ in Corollary~\ref{cor:homotopyofsingularsets}.

To construct the splitting, we proceed as follows. Consider the following diagram
\[
\xymatrix{
H_{n}(BGL_{n}(A)) \ar[r]\ar[d] & H_{n}(BGL_{n+1}(A)) \ar[d]\ar[r] & \cdots \ar[r] & H_{n}(BGL(A)) \ar[d] \\
H_{n}(\Singaone BGL_n(A)) \ar[r] & H_{n}(\Singaone BGL_{n+1}(A))\ar[r] & \cdots \ar[r] & H_{n}(\Singaone BGL(A)),
}
\]
where the horizontal maps are the stabilization maps and the vertical maps are induced by the map from a simplicial presheaf to its singular construction.

The maps in the first row may be determined by appealing to \cite{Suslin} and \cite{NesterenkoSuslin} or \cite{Guin}. In particular, they prove (see \cite[Theorem 3.4(c)]{Suslin} and \cite[Theorem 3.25]{NesterenkoSuslin} \cite[Th\'eor\`eme 1]{Guin}) that for any local ring $A$ with infinite residue field, the map $H_{n}(BGL_{n}(A)) \to H_{n}(BGL_{n+j}(A))$ is an isomorphism for any integer $j \geq 1$. Thus, all the maps in the top row are isomorphisms.

We claim that the map $BGL(A) \to \Singaone BGL(A)$ is actually a homology isomorphism. To see this, recall that
$\Singaone BGL(A)$ represents Karoubi--Villamayor K-theory of $A$. The homotopy groups of the space $\Singaone BGL(A)$
are precisely the Karoubi--Villamayor K-theory groups \cite[Definition 11.4]{KBook}. Since Karoubi--Villamayor K-theory
coincides with Quillen K-theory for regular rings \cite[Corollary 12.3.2]{KBook}, by appeal to the $\text{plus} =
Q$-theorem \cite[Corollary 7.2]{KBook}, we conclude that $\Singaone BGL(A)$ can be taken as a model for the plus construction of $BGL(A)$ and the claim follows.

Combining these observations, we conclude that the composite map
\[
H_{n}(BGL_{n}(A)) \longrightarrow H_{n}(BGL(A)) \longrightarrow H_{n}(\Singaone BGL(A))
\]
is an isomorphism. On the other hand, using the homology stabilization results for the singular construction we mentioned at the beginning of this proof, we conclude that $H_{n}(\Singaone BGL_{n+1}(A)) \to H_{n}(\Singaone BGL(A))$ is an isomorphism. Therefore, the composite map
\[
H_{n}(BGL_{n}(A)) \longrightarrow H_{n}(BGL_{n+1}(A)) \longrightarrow H_{n}(\Singaone BGL_{n+1}(A))
\]
is also an isomorphism. By commutativity of the left-most square, we conclude that the inclusion map $H_{n}(BGL_{n}) \to H_{n}(\Singaone BGL_{n})$ is injective and provides a splitting of the stabilization map as claimed.
\end{proof}

\begin{rem}
The fact that the homotopy groups of the space $\Singaone BGL(A)$ give Karoubi--Villamayor K-theory that we have taken
as a definition of the latter actually goes back to Rector \cite{Rector}. That the latter definition of
Karoubi--Villamayor K-theory coincides with Quillen K-theory is originally due to Gersten \cite[Theorem
3.13]{GerstenQKV}. The fact that the singular construction of $BGL$ represents algebraic $K$-theory in the
$\aone$-homotopy category originates with Morel \cite[Chapter 3]{MorelTHS}.
\end{rem}

\subsubsection*{Milnor K-theory and the relative Hurewicz theorem}
There is another way to link Milnor K-theory and the stabilization map $\Singaone BGL_{n-1} \to \Singaone BGL_n$ via the relative Hurewicz theorem, which we now describe. Again, assume $k$ is a field, and begin by observing that, if $A$ is an essentially smooth local $k$-algebra, then $\pi_1(\Singaone BGL_{n}(A)) = \gm{}(A)$ for any integer $n \geq 1$; this follows, e.g., from the second point of Corollary \ref{cor:homotopyofsingularsets}. There is an action of $\pi_1(\Singaone BGL_{n-1}(A))$ on the homotopy groups of the homotopy fiber of the map $\Singaone BGL_{n-1}(A) \to \Singaone BGL_n(A)$, i.e., on $\pi_i(\Singaone ({\mathbb A}^n \setminus 0)(A))$. The relative Hurewicz theorem describes the relative homology of the above map in the first non-vanishing degree as a quotient of the relative homotopy group by this action.

\begin{lem}
\label{lem:relativeHurewiczinvariantscomputation}
For any field $k$ and any essentially smooth local $k$-algebra $A$, there is a short exact sequence of the form
\[
H_{n}(\Singaone BGL_{n-1}(A)) \longrightarrow H_{n}(\Singaone BGL_{n}(A)) \stackrel{\delta_n}{\longrightarrow} \K^{\mathrm{M}}_n(A).
\]
\end{lem}

\begin{proof}
It follows from \cite[p. 2590]{AsokFaselThreefolds} and the proof of \cite[Proposition 3.5.1]{AsokFaselpi3a3minus0} that the action of $\gm{}(A)$ on $\pi_{n-1}(\Singaone ({\mathbb A}^n \setminus 0)(A))$ is the standard action of $\gm{}(A)$ on $\K^{\MW}_n(A)$. The quotient of $\K^{\MW}_n(A)$ by the standard action is the quotient by the subgroup of $\eta$-divisible elements and therefore is $\K^{\mathrm{M}}_n(A)$.
\end{proof}

\begin{rem}
Recall that one may extend the definition of Milnor K-theory to local rings (e.g., \cite[\S 3]{NesterenkoSuslin}). The affirmation of the Gersten conjecture for Milnor K-theory for regular local rings containing a field \cite[Theorem 7.1]{Kerz} allows us to conclude that the evident map $K^M_n(A) \to \K^{\mathrm{M}}_n(A)$ is actually an isomorphism, when $A$ is a regular local ring. We use this identification without further mention in what follows.
\end{rem}

\subsection{The Suslin--Hurewicz morphism and the stabilization theorem}
\label{ss:comparison}
Suppose now $k$ is a field, and $A$ is an essentially smooth local $k$-algebra with infinite residue field. We now compare Suslin's morphism $s_n$ to the boundary map $\delta_n$ in the relative Hurewicz theorem of Lemma \ref{lem:relativeHurewiczinvariantscomputation}. In light of Suslin's stabilization theorem we can identify the target of $s_n$ with the relative homology group $H_n(BGL_n(A),BGL_{n-1}(A))$. Thus, functoriality of the singular construction and the relative homology exact sequence in conjunction with Lemma \ref{lem:relativeHurewiczinvariantscomputation} and \cite[Theorem 3.25]{NesterenkoSuslin} yield the following commutative diagram of exact sequences:
\[
\resizebox{6in}{!}{\xymatrix{
H_{n}(BGL_{n-1}(A)) \ar[r]\ar[d] & H_n(BGL_{n}(A)) \ar[r]^-{s_n}\ar[d] & \K^{\mathrm{M}}_n(A) \ar[r]\ar[d] & H_{n-1}(BGL_{n-1}(A)) \ar[d]\ar[r]&\\
H_n(\Singaone BGL_{n-1}(A)) \ar[r] & H_n(\Singaone BGL_n(A)) \ar[r]^-{\delta_n} & \K^{\mathrm{M}}_n(A) \ar[r] & H_{n-1}(\Singaone BGL_{n-1}(A)) \ar[r]& \cdots
}}
\]
and our goal is to study the maps $\K^{\mathrm{M}}_n(A) \to \K^{\mathrm{M}}_n(A)$. The above diagram can be very explicitly analyzed for small values of $n$.

\begin{lem}
\label{lem:casen=1}
The map $BGL_1 \to \Singaone BGL_1$ is the identity map of simplicial presheaves.
\end{lem}

\begin{proof}
By definition, the usual bar model of $BGL_1$ is a simplicial object with $i$-simplices given by $\gm{\times i}$. Since $\underline{\hom}(\Delta^n,\gm{^{\times i}}) = \gm{^{\times i}}$, it follows directly from the definition of $\Singaone$ that the map $BGL_1 \to \Singaone BGL_1$ is the identity map.
\end{proof}

\begin{lem}
\label{lem:casen=2}
If $k$ is a field, and if $A$ is an essentially smooth local $k$-algebra with infinite residue field, then the map $BGL_2(A) \to \Singaone BGL_2(A)$ induces an isomorphism on homology in degrees $\leq 2$.
\end{lem}

\begin{proof}
By appeal to Lemma \ref{lem:casen=1} we conclude that the map $H_1(\Singaone BGL_1(A)) \to H_1(\Singaone BGL_2(A))$ is an isomorphism. It follows from Lemma \ref{lem:relativeHurewiczinvariantscomputation} that the map $\delta_2$ is surjective and that there is a commutative diagram of exact sequences of the form
\[
\xymatrix{
H_2(BGL_1(A)) \ar[r]\ar[d]^-{\sim} & H_2(BGL_2(A)) \ar[r]^-{s_2}\ar[d] & \K^{\mathrm{M}}_2(A) \ar[r] \ar[d] & 0\\
H_2(\Singaone BGL_1(A)) \ar[r] & H_2(\Singaone BGL_2(A)) \ar[r]^-{\delta_2} & \K^{\mathrm{M}}_2(A) \ar[r]& 0
}
\]
Now the map $H_2(BGL_2(A)) \to H_2(\Singaone BGL_2(A))$ is split injective by Lemma \ref{lem:homologystabilizationsingularconstruction}. Therefore, the composite of this splitting and the map $H_2(BGL_2(A)) \to \K^{\mathrm{M}}_2(A)$ is surjective. A diagram chase thus implies that the vertical map $\K^{\mathrm{M}}_2(A) \to \K^{\mathrm{M}}_2(A)$ is also surjective. Then, the five lemma implies that $H_2(BGL_2(A)) \to H_2(\Singaone BGL_2(A))$ must be surjective as well and so it is an isomorphism.
\end{proof}

\subsubsection*{Comparison of stabilization maps}
To analyze the comparison maps in general, we will make use of an auxiliary space, which we now describe. Following Schlichting \cite[\S 6]{Schlichting}, for a $k$-scheme $U$, we define $\tilde{E}_n(U)$ to be the maximal perfect subgroup of the kernel of the map $GL_n(\Gamma(U,{\mathscr O}_U)) \to \pi_0(\Singaone GL_n(\Gamma(U,{\mathscr O}_U)))$. By construction, $\tilde{E}_n(U)$ is a subgroup of $SL_n(U)$ since it maps to zero in the commutative group $GL_n(U)/SL_n(U) = {\mathscr O}_U(U)^{\times}$. If either $n \geq 3$ or $n=2$ and the residue fields of smooth $k$-schemes have $> 3$ elements, then the inclusion of presheaves $\tilde{E}_n \to SL_n$ becomes an isomorphism after Zariski sheafification \cite[Lemma 6.5]{Schlichting} and $\tilde{E}_n(U)$ is a presheaf of perfect groups. We write $BGL_n^+$ for the simplicial presheaf obtained by appeal to the functorial version of the plus construction \cite[Chapter VII \S 6]{BousfieldKan} applied to $BGL_n$ and the presheaf of perfect groups $\tilde{E}_n$. The next result is essentially contained in \cite[Corollary 6.16]{Schlichting} and the discussion immediately preceding that statement.

\begin{prop}
\label{prop:laoneplussplitting}
Suppose $k$ is an infinite field. If $n \geq 1$ is an integer, then the following statements hold.
\begin{enumerate}[noitemsep,topsep=1pt]
\item The map $BGL_n \to \Singaone BGL_n$ factors through $BGL_n^+$ and these factorizations are functorial in $n$.
\item The map $\Laone BGL_n \to \Laone BGL_n^+$ is split in the simplicial homotopy category.
\end{enumerate}
\end{prop}

\begin{proof}
The first statement is immediate from the definition of $\tilde{E}_n$ and the definition of the plus construction. For the second statement, observe that $BGL_n \to \Singaone BGL_n$ is an $\aone$-weak equivalence, and the target is $\aone$-local by \cite{AHWII}. Therefore, this map becomes a simplicial weak equivalence after $\aone$-localization. The second point then follows immediately from the first.
\end{proof}

\begin{prop}
Suppose $k$ is an infinite field. For any $n \geq 2$, the map on simplicial homotopy fibers induced by the square
\[
\xymatrix{
BGL_{n-1}^{+} \ar[r]\ar[d] & BGL_n^{+} \ar[d] \\
\Singaone BGL_{n-1} \ar[r] & \Singaone BGL_n
}
\]
induces an isomorphism on homotopy sheaves in degrees $\leq n-1$.
\end{prop}

\begin{proof}
Suppose $k$ is an infinite field and $A$ is an essentially smooth local $k$-algebra. Write $F_n$ for the homotopy fiber of the map in the top row. By \cite[Theorem 5.38]{Schlichting}, one knows that the presheaves $\bpi_i(F_n)(A)$ vanish for $i \leq n-2$ and have Zariski sheafification $\K^{MW}_n$ for $i = n-1$. Likewise, Proposition \ref{prop:fibersequences} states that the homotopy fiber of the bottom horizontal map is $\Singaone ({\mathbb A}^n \setminus 0)$. Corollary \ref{cor:homotopyofsingularsets} shows that this homotopy fiber is also $(n-2)$-connected and has $(n-1)$-st Zariski homotopy sheaf isomorphic to $\K^{MW}_n$ in degree $i-1$.

To establish the result, it suffices to show that the induced map $\K^{MW}_n \to \K^{MW}_n$ is an isomorphism. By Proposition \ref{prop:laoneplussplitting} it this map is actually split injective. The endomorphism ring of $\K^{MW}_n$ coincides with $(\K^{MW}_n)_{-n}(k) \cong \GW(k)$ for any $n \geq 1$ by Lemma \ref{lem:homfromkmwiscontraction}. Since the ring $GW(k)$ contains no idempotents besides $0$ or $1$ by \cite[Theorem 3.9]{KRW} (see Example 3.11 for details), it follows that our map $\K^{MW}_n \to \K^{MW}_n$ is an isomorphism, as required.
\end{proof}

\begin{lem}
\label{lem:fundamentalgroupplus}
Suppose $k$ is an infinite field. For any integer $n \geq 1$, the map of presheaves $\bpi_1(BGL_n^+) \to \bpi_1(\Singaone BGL_n)$ is an isomorphism after Zariski sheafification.
\end{lem}

\begin{proof}
For a commutative ring $R$, we know that $\pi_1(BGL_n^+(R)) = GL_n(R)/\tilde{E}_n(R)$ by definition, and \cite[Lemma 6.5]{Schlichting} implies that the Zariski sheafification of $\tilde{E}_n$ coincides with that of $SL_n$. In particular, the map $GL_n/\tilde{E}_n \to \gm{}$ given by the determinant is an isomorphism after Zariski sheafification. Likewise, the Zariski sheafification of $\bpi_1(\Singaone BGL_n)$ coincides with $\gm{}$. Now the map in question factors through the natural map $\bpi_1(BGL_n) \to \bpi_1(\Singaone BGL_n)$ by Proposition \ref{prop:laoneplussplitting}. However, the map $\bpi_1(BGL_n) \to \bpi_1(\Singaone BGL_n)$ is the map sending $GL_n(R)$ to the quotient by the subgroup of matrices homotopic to the identity, and the result follows.
\end{proof}

\subsubsection*{Refined weak homotopy invariance results}
\begin{thm}
\label{thm:maintechnical}
If $k$ is an infinite field, and $A$ is an essentially smooth local $k$-algebra then the following statements hold.
\begin{enumerate}[noitemsep,topsep=1pt]
\item There is a commutative diagram having exact rows:
\[
\xymatrix{
H_{n}(BGL_{n-1}(A)) \ar[r]\ar[d] & H_n(BGL_{n}(A)) \ar[r]\ar[d] & \K^{\mathrm{M}}_n(A) \ar[r]\ar[d] & 0 \\
H_n(\Singaone BGL_{n-1}(A)) \ar[r] & H_n(\Singaone BGL_n(A)) \ar[r]^-{\delta_n} & \K^{\mathrm{M}}_n(A) \ar[r] & 0,
}
\]
in particular, $\delta_n$ is surjective.
\item The maps $H_i(\Singaone BGL_n(A)) \to H_i(\Singaone BGL_{n+1}(A))$ are isomorphisms for $i \leq n$.
\item The maps $H_i(BGL_n(A)) \to H_i(\Singaone BGL_n(A))$ are isomorphisms when $i \leq n$.
\end{enumerate}
\end{thm}

\begin{proof}
The diagram in question arises from the long exact sequence in relative homology. The maps $BGL_n \to \Singaone BGL_n$
factor through $BGL_n^+$ compatibly with $n$ by Proposition \ref{prop:laoneplussplitting}.  Observe that $H_n(BGL_n^+(A),BGL_{n-1}^+(A)) \cong \K^{\mathrm{M}}_n(A)$ by \cite[Theorem 5.38]{Schlichting} and Lemma \ref{lem:fundamentalgroupplus} and the same relative Hurewicz argument as in Lemma \ref{lem:relativeHurewiczinvariantscomputation}. Moreover, the induced map
\[
H_n(BGL_n^+(A),BGL_{n-1}^+(A)) \longrightarrow H_n(\Singaone BGL_n(A),\Singaone BGL_{n-1}(A))
\]
is an isomorphism. In other words, there is a commutative diagram of exact sequences of the form:
\[
\xymatrix{
H_{n}(BGL_{n-1}^+(A)) \ar[r]\ar[d] & H_n(BGL_{n}^+(A)) \ar[r]\ar[d] & \K^{\mathrm{M}}_n(A) \ar[r]\ar[d] & 0 &\\
H_n(\Singaone BGL_{n-1}(A)) \ar[r] & H_n(\Singaone BGL_n(A)) \ar[r]^-{\delta_n} & \K^{\mathrm{M}}_n(A) \ar[r] & 0.
}
\]
where the map $\K^{\mathrm{M}}_n(A) \to \K^{\mathrm{M}}_n(A)$ is an isomorphism. Since the maps $BGL_n(A) \to BGL_n^+(A)$ are homology isomorphisms, the result follows. It follows immediately that $\delta_n$ is surjective.

The fact that $\delta_n$ is surjective implies the map $H_n(\Singaone BGL_n(A)) \to H_n(\Singaone BGL_{n+1}(A))$ is an isomorphism, i.e., Point (2). Thus, we conclude that the stabilization map $H_n(\Singaone BGL_n(A)) \to H_n(\Singaone BGL(A))$ is an isomorphism as well. In the proof of Lemma \ref{lem:homologystabilizationsingularconstruction}, we established that the map $H_n(BGL_n(A)) \to H_n(\Singaone BGL(A))$ is an isomorphism, and this map factors through the map $H_n(BGL_n(A)) \to H_n(\Singaone BGL_n(A))$, which we therefore also conclude is an isomorphism; thus Point (3) is established.
\end{proof}

\begin{rem}
If the map $H_i(BGL_n(A)) \to H_i(\Singaone BGL_n(A))$ is an isomorphism in degrees $\leq n$, we say weak homotopy invariance holds for the (integral) homology of $GL_n$ in degrees $\leq n$ \cite[Definition 2.3]{HutchinsonWendt}. The notion of weak homotopy invariance was introduced by F. Morel in his approach to the Friedlander--Milnor conjecture (cf. \cite{MFM}, though this notion does not explicitly appear there). In particular, Theorem \ref{thm:maintechnical} implies that weak homotopy invariance holds in degrees $\leq n$ for $GL_n$ over local rings. In contrast, if one replaces $GL_n$ by $SL_2$, one knows that weak homotopy invariance {\em fails} in degree $3$, e.g., by \cite[Theorem 1]{HutchinsonWendt}.
\end{rem}

\subsubsection*{The comparison theorem}
\begin{thm}
\label{thm:suslinmorphismscoincide}
Assume $k$ is an infinite field, $n \geq 1$ is an integer, and $A$ is an essentially smooth local $k$-algebra.
\begin{enumerate}[noitemsep,topsep=1pt]
\item There is an exact sequence of the form
\[
\K^Q_n(A) \longrightarrow \K^{\mathrm{M}}_n(A) \longrightarrow \mathbf{S}_n(A) \longrightarrow 0,
\]
where the leftmost map is the Suslin--Hurewicz homomorphism.
\item Suslin's conjecture holds in degree $n$ if and only if the canonical surjection $\K^{\mathrm{M}}_n/(n-1)!(A) \to \mathbf{S}_n(A)$ is an isomorphism.
\end{enumerate}
\end{thm}

\begin{proof}
For Point (1), recall from the proof of Lemma \ref{lem:homologystabilizationsingularconstruction} that the map $BGL^+(A) \to \Singaone BGL(A)$ is a weak equivalence. On the other hand, Theorem \ref{thm:maintechnical}, and functoriality of the Hurewicz map yield a commutative diagram of the form
\[
\resizebox{6in}{!}{\xymatrix{
\pi_n(BGL^+(A)) \ar[r]\ar[d] & H_n(BGL^+(A)) \ar[r]\ar[d] & H_n(BGL_n(A)) \ar[r]\ar[d] & \K^{\mathrm{M}}_n(A) \ar[r]\ar[d] & 0 \\
\pi_n(\Singaone BGL(A)) \ar[r] & H_n(\Singaone BGL(A)) \ar[r]& H_n(\Singaone BGL_n(A)) \ar[r]^-{\delta_n}& \K^{\mathrm{M}}_n(A) \ar[r] & 0
}}
\]
where all the vertical morphisms are isomorphisms. Thus, the Suslin--Hurewicz morphism coincides with the composite in the bottom row under these isomorphisms; we call this composite $\psi'_n$.

Next, we claim that the composite in the bottom row coincides with the morphism $\psi_n$. Indeed, this is essentially a consequence of functoriality of the relative Hurewicz morphism. To this end, observe that the stabilization map $\pi_n(\Singaone BGL_{n+1}(A)) \cong \pi_n(\Singaone BGL(A))$ is an isomorphism by Lemma \ref{lem:homologystabilizationsingularconstruction}. The relative Hurewicz map $\pi_{n-1}(\Singaone ({\mathbb A}^n \setminus 0)(A)) \to H_{n-1}(\Singaone BGL_n(A),\Singaone BGL_{n-1}(A))$ coincides with the standard quotient map $\K^{MW}_n(S) \to \K^{\mathrm{M}}_n(S)$ by the proof of Lemma \ref{lem:relativeHurewiczinvariantscomputation}. By functoriality of (relative) Hurewicz maps, we therefore have a commutative square of the form
\[
\xymatrix{
\pi_n(\Singaone BGL_n(A)) \ar[r]\ar[d] & \K^{MW}_n(A) \ar[d] \\
H_n(\Singaone BGL_n(A)) \ar[r] & \K^{\mathrm{M}}_n(A).
}
\]
Now, the map $\psi_n$ is defined by factorization through the stabilization map $\pi_n(\Singaone BGL_n(A)) \to \pi_{n}(\Singaone BGL_{n+1}(A))$. Using the splitting of Lemma \ref{lem:homologystabilizationsingularconstruction}, it follows immediately that the map \\$H_n(\Singaone BGL_n(A)) \to \K^{\mathrm{M}}_n(A)$ also factors through the stabilization map
\[
H_n(\Singaone BGL_n(A)) \to H_n(\Singaone BGL_{n+1}(A)).
\]
Combining these observations and appealing again to functoriality of Hurewicz maps, we conclude that $\psi_n$ coincides with $\psi'_n$.

Point (2) follows immediately from Point (1).
\end{proof}


\section{The sheaf $\pi_4^{\aone}(S^{3+3\alpha})$ revisited}
\label{s:pi4refinement}
In the previous section, we showed that Suslin's conjecture from the introduction was equivalent to showing that the natural epimorphism $\K^{\mathrm{M}}_n \to \mathbf{S}_n$ induced an isomorphism $\K^{\mathrm{M}}_n/(n-1)! \to \mathbf{S}_n$. In this section, we verify this assertion in the case $n = 5$. To do this, we proceed in two steps.

We refine \cite[Theorem 5.2.5]{AWW}, which describes the $\aone$-homotopy sheaf $\bpi_4^{\aone}({\pone}^{\sma 3})$ under certain restrictions on the base field in two ways. By analyzing real and $\ell$-adic realizations, we show that most of the restrictions on the base field in the statement of \cite[Theorem 5.2.5]{AWW} are superfluous. The outcome of this analysis is a description of the sheaf $\bpi_4^{\aone}({\pone}^{\sma 3})$ that parallels that of $\bpi_n^{\aone}(BGL_n)$ in \S \ref{ss:snandhomologyofsing}: it may be described as an extension of a sheaf defined in terms of higher Grothendieck--Witt theory by a ``non-stable" part, which is related to the sheaf $\mathbf{S}_5$ analyzed in the previous section (though we caution the reader that ``non-stable" is in quotes here to distinguish it from analyzing what happens after $\pone$-stabilization, which will also be of interest to us).

Second, by comparing our refined description to the stable results of \cite{RSO}, we are able to conclude that $\mathbf{S}_5$ coincides with $\K^{\mathrm{M}}_5/24$. The main computational result is achieved in Theorem \ref{thm:maincomputation}. We conclude with Theorem \ref{thm:suslinsconjecturedegree5}, which shows Suslin's conjecture holds in degree $5$ under appropriate hypotheses.

\subsection{The degree map and realization}
\label{ss:realization}
In this section, we establish some preliminary results about real realization of homotopy groups and, in particular, the interaction between real realization and module structures on $\aone$-homotopy sheaves.

\subsubsection*{Compositions in $\aone$-homotopy groups}
Fix an integer $p \geq 2$. For any integer $q \geq 0$ we may consider the motivic sphere $S^{p+q\alpha}$. The abelian group $[S^{p+q\alpha},S^{p+q\alpha}]_{\aone}$ of $\aone$-homotopy endomorphisms of $S^{p+q\alpha}$ admits a natural ring structure via composition. If $q = 0$, this ring is isomorphic to $\Z$ since the simplicial $p$-sphere $S^p$ is already $\aone$-local and thus the $\aone$-homotopy classes of maps coincide with those in the simplicial homotopy category. If $q > 0$, this ring is isomorphic to $\K^{\MW}_0(k)$ by Morel's computations. More generally, we recall the following result of Morel's.

\begin{lem}[{\cite[Corollary 6.43]{MField}}]
If $p \geq 2$ is an integer and $r,q \geq 0$ are integers, then
\[
[S^{p+q\alpha},S^{p+r\alpha}] = \begin{cases}\K^{\MW}_{r-q}(k) & \text{ if } r > 0 \\ \Z & \text{ if } r = q = 0 \\ 0 & \text{ otherwise}. \end{cases}
\]
\end{lem}

Note that $\gm{}$-suspension induces group homomorphisms
\[
[S^{p+q\alpha},S^{p+r\alpha}]_{\aone} \longrightarrow [S^{p+(q+1)\alpha},S^{p+(r+1)\alpha}]_{\aone}.
\]
If $q,r > 0$, then these maps are isomorphisms. If $r = q = 0$, then the map is the unit map $\Z \to \K^{\MW}_0(k)$, while if $r = 0$ and $q \neq 0$, the map is the zero map.

Given a third integer $s \geq 0$, composition yields homomorphisms
\[
\begin{split}
 [S^{p+r\alpha},S^{p+s\alpha}]_{\aone} \times [S^{p+q\alpha},S^{p+r\alpha}]_{\aone} &\longrightarrow [S^{p+q\alpha},S^{p+s\alpha}]_{\aone} \\
(f,g) &\longmapsto f \circ g.
\end{split}
\]
These maps are bilinear and associative and compatible with $\gm{}$-suspension as described above. In particular, we obtain an ${\mathbb N} \times {\mathbb N}$-graded ring structure on $\bigoplus_{q,r \geq 0} [S^{p+q\alpha},S^{p+r\alpha}]_{\aone}$.

\subsubsection*{Real realization of endomorphism rings}
If $k$ is a field and $k \hookrightarrow \real$ is an embedding, then sending a smooth $k$-scheme $X$ to the topological space $X(\real)$ equipped with its usual structure of a real manifold extends to a functor
\[
{\mathfrak R}: \hop{k} \longrightarrow \mathscr{H}_{\bullet}.
\]
At the level of homotopy categories, this functor was exposed in \cite[\S 3 pp. 121-122]{MV} and was later described in terms of a Quillen adjunction in \cite[\S 5.3]{DuggerIsaksen}. It follows from the construction of \cite{DuggerIsaksen} that real realization is a simplicial functor of simplicial model categories.

Since the real realization of $\gm{}$ is homotopy equivalent to $S^0$, it follows that the real realization of $S^{p+q\alpha}$ is $ S^p$ for any $q \geq 0$. Therefore, by functoriality, real realization induces a ring homomorphism
\[
{\mathfrak R}: \bigoplus_{q,r \geq 0} [S^{p+q\alpha},S^{p+r\alpha}]_{\aone} \longrightarrow [S^p,S^p] = \Z.
\]
We now analyze this ring homomorphism; we begin with the degree $(0,0)$-part.

Following \cite[p. 53]{MField}, set $\epsilon = - \langle -1 \rangle \in \K^{\MW}_0(k)$ . The element $\epsilon$ represents the $\aone$-homotopy class of the endomorphism of $S^{2+2\alpha}$ obtained by switching the two factors of $\gm{}$ under the identification $[S^{2+2\alpha},S^{2+2\alpha}]_{\aone} \cong \K^{\MW}_0(k)$ \cite[Lemma 6.1.1(2)]{MIntro}.

The next result is a straightforward consequence of the description of the Grothendieck--Witt ring of the real numbers combined with the fact that the real realization of $\epsilon$ is the identity map $S^2 \to S^2$.

\begin{lem}
\label{lem:realrealizationdegree00}
For any integers $p,q$, $p \geq 2, q \geq 0$, the ring map $[S^{p+q\alpha},S^{p+q\alpha}] \to [S^p,S^p]$ induced by real realization is an isomorphism if $q = 0$ and the surjection $\Z[\epsilon]/(\epsilon^2 - 1) \to \Z$ given by evaluation at $1$ if $q \neq 0$.
\end{lem}

The groups $\K^{\MW}_q(\real)$ are generated by the expressions $[a]$ for $a \in \real^{\times}$ and an element $\eta$. Therefore, to understand real realization more generally, we need to understand the real realizations of these elements.

\begin{prop}
\label{prop:realizationofelements}
The following statements hold about real realization:
\begin{enumerate}[noitemsep,topsep=1pt]
\item For any $a \in \real^{\times}$,
\[
\mathfrak{R}([a]) = \begin{cases}0 & \text{ if } a > 0 \text{, and} \\ 1 & \text{ if } a < 0; \end{cases}
\]
\item For any $a \in \real^{\times}$,
\[
\mathfrak{R}(\langle a \rangle) = \begin{cases}1 & \text{ if } a > 0 \text{, and} \\ -1 & \text{ if } a < 0; \text{ and } \end{cases}
\]
\item If $\eta: S^{1+2\alpha} \to S^{1+\alpha}$ is the motivic Hopf map, then ${\mathfrak R}(\eta) = -2$.
\end{enumerate}
\end{prop}

\begin{proof}
For the first point, by definition $[a]$ is the stabilization of the map $S^0_k \to \gm{}$ sending the non-base-point of $S^0_k$ to the element $a \in \gm{}(k)$. It follows that if $a > 0$, then the real realization of this map is the constant map, which has degree $0$. Likewise, if $a < 0$, then the real realization of this map is the identity map $S^0 \to S^0$, whose suspensions all have degree $1$.

For the second point, consider the element $\langle a \rangle \in \K^{MW}_0(\real)$. If $a > 0$, then $\langle a \rangle \cong \langle 1 \rangle = 1$ and is sent to $1$ under real realization. If $a < 0$, then $\langle a \rangle \cong \langle -1 \rangle$. Since ${\mathfrak R}(\epsilon) = 1$ by Lemma \ref{lem:realrealizationdegree00}, it follows that ${\mathfrak R}(\langle -1 \rangle) = -1$.

Finally, since $\langle -1 \rangle = 1 + \eta[-1]$ and real realization is a ring homomorphism, we conclude that $\mathfrak{R}(\eta) = -2$ from the discussion above.
\end{proof}

Given these preliminary results, we may state the main result about realization we will use.

\begin{prop}
\label{prop:realrealizationofdegreemaps}
Suppose $n \geq 2$ is an integer. For any integer $j > 0$, the image of the homomorphism given by real realization
\[
\operatorname{im}(\bpi_{{n-1}+j\alpha}^{\aone}(S^{n-1 + n\alpha})(\real) \to \pi_{n-1}(S^{n-1})) = \begin{cases} (1) & \text{ if } j < n;\\ (2^{n-j}) & \text{ if } j \geq n \end{cases}.
\]
More precisely, for $j = n$ the map factors as $\mathbf{GW}(\real) \to \mathbf{W}(\real) \stackrel{sgn}{\to} \Z$.
\end{prop}

\begin{proof}
In every case, the source group is $\K^{\MW}_{n-j}(\real)$. If $j < n$, this group contains the $(n-j)$-fold product $[-1] \cdots [-1]$ and the result follows immediately from Proposition \ref{prop:realizationofelements}(1) and compatibility of real realization with the product structures described above.

If $j = n$, the result follows from Lemma \ref{lem:realrealizationdegree00}. If $j > n$, then the source of the map is the free abelian group generated by $\eta^{j-n}$ and the result follows from Proposition \ref{prop:realizationofelements}(3).
\end{proof}

\subsubsection*{\'Etale realization}
We refer the reader to \cite{Isaksen} for a detailed discussion of \'etale realization. Let us quickly summarize the
main points. If $\ell$ is prime, one begins by defining an $\ell$-complete \'etale realization functor on the category of schemes: given a scheme $X$, its \'etale realization is an $\ell$-complete pro-simplicial set that we will denote by $\Et(X)$. The construction has the property that a morphism of schemes $f\colon X\to Y$ induces a weak equivalence $\Et(X)\to\Et(Y)$ if and only if $f^*\colon H^*_{\et}(Y;\Z/(\ell))\to H^*_{\et}(X;\Z/(\ell))$ is an isomorphism. By \cite{Isaksen}, if $k$ is a field and $\ell$ is different from the characteristic of $k$, the assignment $X \mapsto \Et(X)$ on smooth $k$-schemes extends to a functor on the pointed $\aone$-homotopy category $\hop{k}$; abusing notation slightly, we will also denote this functor by $\Et$. If $k$ is furthermore separably closed, it follows from the K\"unneth isomorphism in \'etale cohomology with $\Z/(\ell)$-coefficients that the functor $\Et$ preserves finite products and smash products of pointed spaces. Moreover, by construction, the functor $\Et$ commutes with the formation of homotopy colimits.

Assume now that $k$ is separably closed and write $R$ for the ring of Witt vectors in $k$. Choose an algebraically closed field $K$ and fix embeddings $R \hookrightarrow K$ and $\cplx \hookrightarrow K$. For any split reductive group $G$, these morphism yield maps of the form:
\[
G_k \longrightarrow G_R \longleftarrow G_K \longrightarrow G_{\cplx}
\]
which we will use to compare the \'etale realization over $k$ with complex realization.

\begin{lem}
\label{lem:etalerealizationofspheres}
If $k$ is a separably closed field having characteristic $p$ and $\ell$ is a prime different from $p$, then for any integers $i,j \geq 0$,
\[
\Et(S^{i+j\alpha}) \cong (S^{i+j})^{\wedge}_{\ell}.
\]
\end{lem}

\begin{proof}
The comparison maps described before the statement applied to $G = \gm{}$ yield identifications $\Et(\gm{}) \cong (\gm{}(\cplx))^{\wedge}_{\ell} \cong (S^1)^{\wedge}_{\ell}$. The \'etale realization of the simplicial circle is also $(S^1)^{\wedge}_{\ell}$ since \'etale realization preserves homotopy colimits and the simplicial circle can be realized as a homotopy pushout of the diagram $\ast \leftarrow S^0_k \rightarrow \ast$. To conclude, we use the fact that \'etale realization preserves smash products of pointed spaces.
\end{proof}

\begin{rem}
In contrast to the case of real realization, over a separably closed field, the \'etale realization of the endomorphisms of the motivic sphere are determined wholly by the realization of the identity map and the realization of $\eta$.
\end{rem}

\subsection{Realization of some homotopy sheaves of spheres}
\label{ss:realrealizationhomotopysheaves}
The purpose of this section is to describe the behavior of various homotopy sheaves of spheres under real and \'etale realization. In particular, we analyze the $\real$-realization homomorphism $\bpi_{3+j\alpha}^{\aone}(S^{2+3\alpha})(\real) \to \pi_3(S^2)$ and correct \cite[Corollary 5.4.1]{AsokFaselpi3a3minus0}. Along the way, we remind the reader of the computation of this $\aone$-homotopy sheaf, which will appear in subsequent sections.

\subsubsection*{The computation of $\bpi_{3+j\alpha}^{\aone}(S^{2+3\alpha})$ revisited}
We begin by recalling some results about $\aone$-fiber sequences from \cite[\S 4.2]{AsokFaselpi3a3minus0}. First, for any integer $n \geq 1$ there is a pull-back diagram of linear algebraic groups
\[
\xymatrix{
Sp_{2n} \ar[r]\ar[d] & SL_{2n+1} \ar[d] \\
Sp_{2n+2} \ar[r] & SL_{2n+2}.
}
\]
This diagram yields isomorphisms of $k$-schemes $SL_{2n+2}/Sp_{2n+2} \cong SL_{2n+1}/Sp_{2n}$. If we define $X_n = SL_{2n}/Sp_{2n}$, then by \cite[Proposition 4.2.2]{AsokFaselpi3a3minus0} we obtain $\aone$-fiber sequences of the form
\begin{equation}
\label{eq:1}
X_n \longrightarrow X_{n+1} \longrightarrow {\mathbb A}^{2n+1} \setminus 0.
\end{equation}
In the case $n=2$, there is an exceptional isomorphism $X_2 = SL_4/Sp_4 \cong SL_3/SL_2$ and thus $X_2$ is
$\aone$-weakly equivalent to ${\mathbb A}^3 \setminus 0$, yielding an $\aone$-fiber sequence
\[
 S^{2 +3\alpha} \weq {\mathbb A}^3 \setminus 0 \longrightarrow X_3 \longrightarrow {\mathbb A}^5 \setminus 0 \weq S^{4
 + 5 \alpha}.
\]

The schemes $X_n$ are symmetric varieties, and $X_{\infty} := \colim_{n \geq 0} X_n$ is a model for the $\aone$-connected component of the base-point in the space $GL/Sp$ arising in higher Grothendieck--Witt theory \cite{SchlichtingTripathi}. In particular, as observed in \cite[Proposition 4.2.2]{AsokFaselpi3a3minus0}, $\bpi_i^{\aone}(X_n) \cong \mathbf{GW}^3_{i+1}$ for $i \leq 2n-2$, at least if we work over a field having characteristic unequal to $2$. Using the long exact sequence in $\aone$-homotopy sheaves associated with the above fiber sequence, there is an associated exact sequence of the form
\[
\mathbf{GW}^3_5 \longrightarrow \K^{MW}_5 \longrightarrow \bpi_{3}^{\aone}(S^{2+3\alpha}) \longrightarrow \mathbf{GW}^3_4 \longrightarrow 0.
\]
The cokernel of the morphism $\mathbf{GW}^3_5 \to \K^{MW}_5$ is called $\mathbf{F}_5$ in \cite[Theorem 4.3.1]{AsokFaselpi3a3minus0}, and it is observed that $\mathbf{F}_5$ is a quotient of the fiber product $\mathbf{S}_5 \times_{\K^{\mathrm{M}}_5/2} \mathbf{I}^5 =: \mathbf{T}_5$ from \cite[Theorem 3.14]{AsokFaselSpheres}.

\subsubsection*{Behavior of $\bpi_{3+j\alpha}^{\aone}(S^{2+3\alpha})$ under realization}
There are natural homeomorphisms $X_n(\real) \approx SL_{2n}(\real)/Sp_{2n}(\real)$. Replacing the groups by their homotopy equivalent maximal compact subgroups, we obtain weak equivalences of the form
\[
X_n(\real) \weq SO(2n)/U(n).
\]
Furthermore, the fiber sequence \eqref{eq:1} corresponds under $\real$-realization to the fiber sequence
\begin{equation}
 \label{eq:2}
 X_n(\real) \longrightarrow X_{n+1}(\real) \longrightarrow S^{2n}.
\end{equation}
While real realization need not necessarily preserve fiber sequences, it follows from this observation that the real realization of $X_n \to X_{n+1} \to {\mathbb A}^{2n+1}\setminus 0$ is sent to a fiber sequence by real realization.

In a range depending on $n$, the homotopy groups of $X_{n+1}(\real)$ may be computed by means of real Bott periodicity. For instance, the stabilization map
\[
\pi_3(X_3(\real)) = \pi_3(SO(6)/U(3)) \longrightarrow \pi_3(SO/U) = \pi_3(O/U) = 0
\]
is an isomorphism. We therefore obtain, from a portion of the long exact homotopy sequence of \eqref{eq:2}, a
presentation
\begin{equation*}
 \pi_4(X_3(\real)) \to \pi_4(S^4) \to \pi_3(S^2) \to 0.
\end{equation*}
Since $\pi_4(S^4) \cong \pi_3(S^2) \cong \Z$, we conclude that the map $\pi_4(S^4) \to \pi_3(S^2)$ is an isomorphism.


Let $j \ge 0$. One obtains a ladder diagram from the long exact homotopy sequences of \eqref{eq:1} and \eqref{eq:2} and
the $\real$ realization map ${\mathfrak R}$:
\begin{equation}
 \label{eq:3}
 \xymatrix{ \ar[r] & \bpia_{4+j\alpha}(S^{4+5\alpha})(\real) \ar^f[r] \ar^{\rho'}[d] & \bpia_{3+j\alpha}(S^{2+3\alpha})(\real) \ar[r] \ar^{\rho}[d] &
 \bpia_{3+j\alpha}(X_3)(\real) \ar[r] \ar[d] & 0 \\ \ar[r] & \pi_4(S^4) \ar^{\iso}[r] & \pi_3(S^2) \ar[r] & 0 \ar[r] &
 0. }
\end{equation}
According to \cite[Proposition 5.2.1]{AsokFaselpi3a3minus0}, $f$ is surjective when $j \ge 5$, and is an isomorphism
when $j \ge 6$. The next result describes the image of the middle map and corrects \cite[Corollary 5.4.1]{AsokFaselpi3a3minus0}.

\begin{prop}
\label{prop:realrealizationofpi3a3minus0}
For any integer $j \geq 0$, the image of the homomorphism given by $\real$-realization satisfies
\[
\operatorname{im}(\bpi_{3+j\alpha}^{\aone}(S^{2+3\alpha})(\real) \to \pi_3(S^2)) = \begin{cases} (1) & \text{ if } j < 5 \\ (2^{j-5}) & \text{ if } j \geq 5.\end{cases}
\]
\end{prop}

\begin{proof}
  In view of the commutativity of Diagram \ref{eq:3}, and the observation that $f$ is an isomorphism when $f \ge 6$,
  this result is an immediate consequence of the corresponding fact for $\rho'$, which is contained in Proposition
  \ref{prop:realrealizationofdegreemaps}.
\end{proof}

\begin{rem}
  We remind the reader that $S^{2+3 \alpha}$ is $\aone$-weakly equivalent to $\mathbb{A}^3 \setminus \{0\}$.  It is
  known that $\pi_{3+5\alpha}(S^{2 + 3 \alpha}) \cong \Z/(24) \times_{\Z/(2)} \mathbf{W}$ as a $\mathbf{GW}$-module
  \cite[Proposition 5.2.1]{AsokFaselpi3a3minus0}, but this will not be needed in the sequel.
\end{rem}

\subsection{The EHP sequence and homotopy sheaves of $S^{3+3\alpha}$}
The sheaf $\bpi_{4+j\alpha}^{\aone}(S^{3+3\alpha})$ is computed from $\bpi_{3+j\alpha}^{\aone}(S^{2+3\alpha})$ by appeal to the simplicial EHP sequence of \cite[Theorem 3.3.13]{AWW}. In particular, for any base field $k$, there is a short exact sequence of the form
\begin{equation}
\label{eqn:ehps23}
\bpi_{5+j\alpha}^{\aone}(S^{5+6\alpha}) \stackrel{\mathrm{P}_k}{\longrightarrow} \bpi_{3+j\alpha}^{\aone}(S^{2+3\alpha}) \stackrel{\mathrm{E}_k}{\longrightarrow} \bpi_{4+j\alpha}^{\aone}(S^{3+3\alpha}) \longrightarrow 0,
\end{equation}
where the morphism $\mathrm{P}$ is induced by composition with the Whitehead square of the identity. By construction, this sequence is natural in $k$. In \cite[Theorem 5.2.5]{AWW}, after some preliminary results, the image of $\mathrm{P}_k$ in Diagram \ref{eqn:ehps23} is identified; our goal is to perform a similar analysis of the EHP sequence here.

By \cite[Corollary 6.43]{MField}, $\bpi_{5+j\alpha}^{\aone}(S^{5+6\alpha}) \cong \K^{\mathrm{MW}}_{6-j}$. On the other hand, the description of $\bpi_{3+j\alpha}^{\aone}(S^{2+3\alpha})$ for large $j$ follows from \cite[Lemma 5.1.1]{AsokFaselpi3a3minus0}: $\bpi_{3+j\alpha}^{\aone}(S^{2+3\alpha}) \cong \mathbf{I}^{6-j}$ for $j \geq 6$ and the latter sheaf coincides with the sheaf $\mathbf{W}$. Granted these two facts, the EHP exact sequence above takes the form:
\begin{equation}
 \label{eq:5}
 \xymatrix{ \bpia_{5+6\alpha} (S^{5 + 6 \alpha}) \ar@{=}[d] \ar^{\mathrm{P}_k}[r] & \bpia_{3+6\alpha} (S^{2 + 3 \alpha}) \ar@{=}[d] \ar^{\mathrm{E}_k}[r] & \bpia_{4+6\alpha} (S^{3 + 3\alpha})
 \ar[r] & 0 \\ \K^{\mathrm{MW}}_0 & \W. }
\end{equation}

\subsubsection*{On the map $\mathrm{P}_k$}
There are identifications $\hom_{\K^{\mathrm{MW}}_0}(\K^{\mathrm{MW}}_0,\W) \cong \mathbf{W}(k) \cong \hom_{\GW(k)}(\GW(k),\W(k))$. Via these identifications, the map $\mathrm{P}_k$ of \eqref{eq:5} determines an element $[\mathrm{P}_k]$ of $\W(k)$. The next result gives some conditions equivalent to the surjectivity of $\mathrm{P}_k$.

\begin{lem}
\label{lem:equivalentconditionsforPsurjective}
The following are equivalent:
\begin{enumerate}[noitemsep,topsep=1pt]
 \item The map $\mathrm{P}_k$ is surjective;
 \item The element $[\mathrm{P}_k]$ generates $\W(k)$ as a $\GW(k)$ module;
 \item The element $[\mathrm{P}_k]$ is a unit of $\W(k)$;
 \item The map $\mathrm{P}_k(k)$ is surjective.
\end{enumerate}
\end{lem}

\begin{proof}
The equivalence of the first and second statements follows from the fact that $\mathrm{P}_k$ is a morphism of $\K^{MW}_0$-modules in conjunction with the identifications 
\[
\hom_{\K^{\mathrm{MW}}_0}(\K^{\mathrm{MW}}_0,\W) \cong \mathbf{W}(k) \cong \hom_{\GW(k)}(\GW(k),\W(k))
\] 
mentioned above. The last identification also immediately implies the equivalence of the second and third statements. The equivalence of the second and fourth statements follows from evaluation on $k$.
\end{proof}

We now prove some descent and ascent results relating the maps $\mathrm{P}_k$ for different classes of fields.

\begin{lem}
\label{lem:descentascentforP}
The following statements hold:
\begin{enumerate}[noitemsep,topsep=1pt]
\item If $K/k$ is a field extension, and if $\mathrm{P}_k$ is surjective, then $\mathrm{P}_K$ is surjective as well.
\item If $k$ is not formally real, $k^s$ is a separable closure of $k$ and $\mathrm{P}_{k^s}$ is surjective, then $\mathrm{P}_k$ is surjective.
\item If $\mathrm{P}_{\real}$ is surjective, then $\mathrm{P}_{\Q}$ is surjective.
\end{enumerate}
\end{lem}

\begin{proof}
In each case, we appeal to the equivalent conditions of Lemma \ref{lem:equivalentconditionsforPsurjective}: the map $\mathrm{P}_k$ is surjective if and only if the element $[\mathrm{P}_k]$ is a unit of $\W(k)$. The latter statement may be checked by appeal to a classical result of Pfister \cite[Theorem 8.7]{Lam}.

For the first point, the definition of $\mathrm{P}_k$ is natural in $k$, which is to say that $[\mathrm{P}_K]$ is the image of $[\mathrm{P}_{k}]$ under the functorial ring map $\W(k) \to \W(K)$. Since the element $[\mathrm{P}_k]$ is a unit, so too is $[\mathrm{P}_K]$.

For the second point we know $\W(k^s) = \ZZ/(2)$, and $\W(k) \to \W(k^s)$ is the dimension map (modulo 2). By \cite[Theorem VIII.8.7]{Lam}, the class $[\mathrm{P}_k]\in \W(k)$ is a unit if and only if it remains a unit in $\W(k^s)$ under this map.

For the third point, begin by observing that the field $\Q$ is formally real and has a unique real closure, the field $\RR^{alg}$ of real algebraic numbers. By \cite[Theorem VIII.8.7]{Lam}; the class $[\mathrm{P}_\Q] \in \W(\Q)$ is a unit if and only if $[\mathrm{P}_{\RR^{alg}}] \in \W(\RR^{alg})$ is a unit. Moreover, a form $q$ is a unit if and only if its signature with respect to the unique ordering of $\RR^{alg}$ is $\pm 1$. Since this last statement can be checked by passing to $\RR$, we conclude.
\end{proof}

\subsubsection*{A vanishing result}
We are now in a position to strengthen \cite[Proposition 5.2.3]{AWW}.

\begin{prop}
\label{prop:pi6contrvanish}
If $k$ is a field having characteristic unequal to $2$, then $\bpi_{4+6\alpha}^{\aone}(S^{3+3\alpha})=0$.
\end{prop}

\begin{proof}
Contemplating the exact sequence in Diagram \eqref{eq:5}, by appeal to Lemmas \ref{lem:equivalentconditionsforPsurjective} and \ref{lem:descentascentforP}(1), it suffices to show that $[\mathrm{P}_k]$ is a unit for $k$ any prime field. We treat the characteristic zero case in Lemma \ref{lem:Qunit} and the positive characteristic case in Lemma \ref{lem:scunit}.
\end{proof}

\begin{lem}\label{lem:Qunit}
The element $[\mathrm{P}_\Q] \in \W(\Q)$ is a unit.
\end{lem}
\begin{proof}
By Lemma \ref{lem:descentascentforP}(3), it is sufficient to show that the map $\mathrm{P}_\real(\real)$ in the following diagram, obtained by evaluating \eqref{eqn:ehps23}, is surjective:
 \[
\xymatrix{ \bpia_{5+6\alpha} (S^{5 + 6 \alpha})(\real) \ar@{=}[d] \ar^-{\mathrm{P}_\real(\real)}[r] & \bpia_{3} (S^{2 + 3 \alpha}) \ar@{=}[d] \ar^-{\mathrm{E}_\real(\real)}[r] & \bpia_{4} (S^{3 + 3\alpha})
 \ar^-{\mathrm{H}_\real(\real)}[r] & 0 \\ \Z[\epsilon]/(\epsilon^2 - 1) & \Z }
 \]
the equality on the left is contained in Lemma \ref{lem:realrealizationdegree00}.

Since $\real$-realization is functorial and the real realization of the relevant portion of the simplicial EHP sequence coincides with the corresponding portion of the EHP sequence of the real points, we obtain a commuting diagram of exact sequences
 \[
 \xymatrix{
 \bpia_{5+6\alpha}(S^{5+6\alpha})(\real) \ar[r]^-{\mathrm{P}}\ar[d] & \bpia_{3+6\alpha}(S^{2+3\alpha})(\real) \ar[r]^-{\mathrm{E}}\ar^-{\rho}[d] & \bpia_{4+6\alpha}(S^{3+3\alpha})(\real) \ar[r]\ar[d] & 0 \\
 \pi_5(S^5) \ar^{|\mathrm{P}|}[r] & \pi_3(S^2) \ar[r] & \pi_4(S^3) \ar[r] & 0.}
 \]
 It is well known that the lower sequence takes the form $\Z\iota_5 \to \Z\eta_{\text{top}} \to \Z/(2) \to 0$, and
 in particular that $|\mathrm{P}|(\iota_5) = 2 \eta_{\text{top}}$.

 The rest of the argument is a diagram chase. Write $\mathrm{P}_\real(1) = n$. We wish to show that $n = \pm 1$. The left hand vertical map is the surjection $\Z[\epsilon]/(\epsilon^2 - 1) \to \Z$ given by evaluation at $1$, again by Lemma \ref{lem:realrealizationdegree00}; it sends the class of the identity map to the class of the identity. Therefore, $2\Z\eta_{\text{top}}$ is the image of the composite map
 $\bpia_{5+6\alpha}(S^{5+6\alpha})(\real) \to \pi_3(S^2)$.

 We showed in Proposition \ref{prop:realrealizationofpi3a3minus0} that the image of the map $\rho$ is $2 \Z \eta_{\text{top}}$. Following the left hand square in the clockwise direction, we see that the image of the composite map
 $\bpia_{5+6\alpha}(S^{5+6\alpha})(\real) \to \pi_3(S^2)$ is $2 n \ZZ \eta_{\text{top}}$. Consequently, $n = \pm 1$ as required.
\end{proof}

\begin{lem}
\label{lem:scunit}
If $k$ is a finite field having characteristic unequal to $2$, then $[\mathrm{P}_k]$ is a unit.
\end{lem}
\begin{proof}
By appeal to Lemma \ref{lem:descentascentforP}(2), it is sufficient to show that the map $\mathrm{P}_{k^s}(k^s)$ is surjective.
\begin{equation}
\label{eq:5b}
\xymatrix{ \bpia_{5+6\alpha} (S^{5 + 6 \alpha})(k^s) \ar@{=}[d] \ar^-{\mathrm{P}_{k^s}(k^s)}[r] & \bpia_{3+6\alpha} (S^{2 + 3 \alpha})(k^s) \ar@{=}[d] \ar^-{\mathrm{E}_{k^s}(k^s)}[r] & \bpia_{4+6\alpha} (S^{3 + 3\alpha})(k^s)
 \ar[r] & 0 \\ \ZZ & \ZZ/ (2). }
\end{equation}
Take $\ell = 2$. By appeal to Lemma \ref{lem:etalerealizationofspheres} and functoriality of \'etale realization, we obtain a commutative diagram of the form:
 \[
 \xymatrix{
 \bpi_{5+6\alpha}^{\aone}(S^{5+6\alpha})(k^s) \ar[r]^-{\mathrm{P}_{k^s}(k^s)}\ar[d] & \bpi_{3+6\alpha}^{\aone}(S^{2+3\alpha})(k^s) \ar[r]^-{\mathrm{E}_{k^s}(k^s)}\ar^{\rho}[d] & \bpi_{4+6\alpha}^{\aone}(S^{3+3\alpha})(k^s) \ar[r]\ar[d] & 0 \\
 \pi_{11}(S^{11})^{\wedge}_2 \ar^{|\mathrm{P}|}[r] & \pi_9(S^5)^{\wedge}_{2} \ar[r] & \pi_{10}(S^6)^{\wedge}_2 \ar[r] & 0.}
 \]
The sequence $\pi_{11}(S^{11}) \to \pi_9(S^5) \to \pi_{10}(S^6) \to 0$ is a portion of the classical EHP exact sequence, and one knows that $\pi_{11}(S^{11}) = \Z$, $\pi_9(S^5) =\Z/(2)$ and $\pi_{10}(S^6) = 0$ and thus this sequence remains exact after $2$-completion, i.e., the diagram above is a commutative diagram of exact sequences. Therefore, we conclude that $|\mathrm{P}|$ is necessarily surjective. Indeed, the identity map lies in the image of the leftmost vertical arrow, and so the composite map $\Z/2 \iso \bpia_{5+6\alpha}(S^{5+6\alpha})(k^s) \to \pi_9(S^5)^{\wedge}_2 \iso \ZZ/(2)$ is surjective, whence it is also an isomorphism. It follows immediately that $\mathrm{P}_{k^s}(k^s)$ is surjective.
\end{proof}

\subsubsection*{Some results on strictly $\aone$-invariant sheaves}
The vanishing result of Proposition \ref{prop:pi6contrvanish} can be used to deduce information about the sheaves $\bpi_{4+j\alpha}^{\aone}(S^{3+3\alpha})$ for $j < 6$ by appeal to some technical results on strongly $\aone$-invariant sheaves that were established in \cite[\S 5.1]{AWW}. We restate the necessary results here for the convenience of the reader; in what follows $\hom$ is taken in the category of Nisnevich sheaves of abelian groups.

\begin{lem}[{\cite[Lemma 5.1.3]{AWW}}]
\label{lem:homfromkmwiscontraction}
Suppose $\mathbf{M}$ is a strictly $\aone$-invariant sheaf.
\begin{enumerate}[noitemsep,topsep=1pt]
\item For any integer $n \geq 1$, there are isomorphisms
\[
\hom(\K^{\MW}_n,\mathbf{M}) \iso \mathbf{M}_{-n}(k).
\]
\item If $n \geq 2$, the evident map $\hom(\K^{\MW}_n,\mathbf{M}) \to \hom(\K^{\MW}_{n-1},\mathbf{M}_{-1})$ induced by contraction is an isomorphism compatible with the identification of \textup{Point (1)}.
\end{enumerate}
\end{lem}

\begin{lem}[{\cite[Lemma 5.1.5]{AWW}}]
\label{lem:contractionsfactor}
Fix a base field $k$. If $\phi: \K^{\MW}_n \to
\mathbf{M}$ is a morphism of sheaves such that
$\phi_{-j} =0$, then
\begin{enumerate}[noitemsep,topsep=1pt]
\item \label{i:cf1} Assuming $n \geq j \geq 0$, the morphism $\phi$ is trivial; and
\item \label{i:cf2} Assuming $0 \leq n < j$, the morphism $\phi$ factors through a morphism $\K^{\MW}_n/\mathbf{I}^{j} \to \mathbf{M}$.
\end{enumerate}
\end{lem}

\subsubsection*{The computation}
Granted the results above, we are now in a position to establish our refinement of \cite[Theorem 5.2.5]{AWW}. Recall from \ref{eqn:ehps23} that there is an exact sequence of the form
\[
 \bpi_{5+j\alpha}^{\aone}(S^{5+6\alpha}) \stackrel{\mathrm{P}_k}{\longrightarrow} \bpi_{3+j\alpha}^{\aone}(S^{2+3\alpha}) \stackrel{\mathrm{E}_k}{\longrightarrow} \bpi_{4+j\alpha}^{\aone}(S^{3+3\alpha}) \longrightarrow 0.
\]
The sheaf $\bpi_{3+j\alpha}^{\aone}(S^{2+3\alpha})$ was described in greater detail at the beginning of Section \ref{ss:realrealizationhomotopysheaves}: for $j = 0$ it is an extension of $\mathbf{GW}^3_4$ by a sheaf called $\mathbf{F}_5$. Since $\bpi_{5}^{\aone}(S^{5+6\alpha}) = \K^{MW}_6$, $\hom(\K^{MW}_6,\mathbf{GW}^3_4) = (\mathbf{GW}^3_4)_{-6}$ by Lemma \ref{lem:homfromkmwiscontraction}, and $(\mathbf{GW}^3_4)_{-6} = 0$ by \cite[Proposition 3.4.3]{AsokFaselpi3a3minus0}, we conclude that (i) the image of the map $\K^{MW}_6 \to \bpi_{3}^{\aone}(S^{2+3\alpha})$ is contained in $\mathbf{F}_5$ and (ii) $\bpi_{3}^{\aone}(S^{2+3\alpha}) \to \mathbf{GW}^3_4$ factors through a map $\bpi_{4}^{\aone}(S^{3+3\alpha}) \to \mathbf{GW}^3_4$.

\begin{prop}
\label{prop:unstablecomputation}
There is an exact sequence
 \begin{equation}
 \label{eq:21}
 \mathbf{S}_5 \longrightarrow \bpi_4^{\aone}(S^{3+3\alpha}) \longrightarrow \GW_4^3 \longrightarrow 0
 \end{equation}
which becomes short exact after $4$-fold contraction.
\end{prop}

\begin{proof}
At the beginning of Section \ref{ss:realrealizationhomotopysheaves} we recalled from \cite[Theorem 4.3.1]{AsokFaselpi3a3minus0} the fact that there is a morphism $\mathbf{T}_5 \to \mathbf{F}_5$ that becomes an isomorphism after $4$-fold contraction. In particular, we conclude that the map $\hom(\K^{MW}_6,\mathbf{T}_5) \to \hom(\K^{MW}_6,\mathbf{F}_5)$ is an isomorphism by Lemma \ref{lem:homfromkmwiscontraction}. Therefore, the map $\K^{MW}_6 \to \mathbf{F}_5$ from the EHP sequence lifts uniquely to a morphism $\K^{MW}_6 \to \mathbf{T}_5$. By definition, $\mathbf{T}_5$ is a fiber product of $\mathbf{S}_5$ and $\mathbf{I}^5$ along $\K^{\mathrm{M}}_5/2$ and there is an exact sequence of the form
\[
0 \longrightarrow \mathbf{I}^6 \longrightarrow \mathbf{T}_5 \longrightarrow \mathbf{S}_5 \longrightarrow 0.
\]
However, since $\hom(\K^{MW}_6,\mathbf{S}_5) = 0$ by combining \cite[Lemma 2.7 and Corollary 3.11]{AsokFaselSpheres}. Therefore, the image of the map $\K^{MW}_6 \to \mathbf{T}_5$ factors through a map $\K^{MW}_6 \to \mathbf{I}^6$.

Now, we appeal to Proposition \ref{prop:pi6contrvanish}. Indeed, it follows from this vanishing result that the map
\[
\bpia_{5+6\alpha}(S^{5+6\alpha}) \longrightarrow \bpia_{3+6\alpha}(S^{2+3\alpha})
\]
is necessarily surjective. Unwinding the definitions, we conclude that the map $\K^{MW}_6 \to \mathbf{I}^6$ of the previous paragraph is surjective, and we obtain the required exact sequence. Exactness after $4$-fold contraction is immediate from \cite[Theorem 4.3.1]{AsokFaselpi3a3minus0}.
\end{proof}



\subsubsection*{Comparison with the stable result}
In order to complete the calculation of $\bpia_4(S^{3+ 3\alpha})$ stated in the introduction, we compare with calculations that have been carried out $\PP^1$-stably in \cite{RSO}. To this end, recall that there is a stabilization map $S^{3+3\alpha} \to \Omega_{\PP^1}^\infty \Sigma_{\PP^1}^{\infty} S^{3+3\alpha}$. By \cite[Theorem 4.4.5]{AsokFaselKO} there are a sequence of morphisms $S^{(n-1)+n\alpha} \to \Omega^{-n}_{\pone}O$ that stabilize to the degree map from the motivic sphere spectrum. As a consequence, the following diagram commutes:
\[
\xymatrix{
\bpi_4^{\aone}(S^{3+3\alpha}) \ar[r]\ar[d] & \mathbf{GW}^3_4 \ar@{=}[d]\\
\bpi_4^{\aone}(\Omega_{\pone}^{\infty}\Sigma_{\pone}^{\infty} S^{3+3\alpha}) \ar[r] & \mathbf{GW}^3_4.
}
\]
In combination with Proposition~\ref{prop:unstablecomputation} we observe that there is an exact sequence of strictly $\aone$-invariant sheaves of the form
\[
\K^{\mathrm{M}}_5/24 \longrightarrow \bpi_4^{\aone}(S^{3+3\alpha}) \longrightarrow \mathbf{GW}^3_4 \longrightarrow 0.
\]
Moreover, both the groups appearing in this extension are stably non-trivial.

Now, suppose $k$ is a field having characteristic exponent $p$ and consider $\Z[\frac{1}{p}]$. Work of R{\"o}ndigs, Spitzweck and {\O}stv{\ae}r \cite[Theorem 5.5]{RSO} establishes a short exact sequence of the form:
\begin{equation}
 \label{eq:24}
 0 \longrightarrow \KM_5/24 \tensor \Z[\frac{1}{p}] \longrightarrow \bpi^{\aone}_4( \Omega_{\PP^1}^\infty \Sigma_{\PP^1}^{\infty} S^{3+3\alpha}) \tensor \Z[\frac{1}{p}] \longrightarrow \GW_4^3 \tensor \Z[\frac{1}{p}]
 \longrightarrow 0.
\end{equation}
We note here that in \cite{RSO} the Hermitian K-theory spectrum is denoted $\mathbf{KQ}$; we have chosen to follow Schlichting's notation since it was consistent with \cite{AWW}. Furthermore, the computation in \cite{RSO} does not look exactly like this, but their statement implies this one because the map out of $\bpi^{\aone}_4( \Omega_{\PP^1}^\infty \Sigma_{\PP^1}^{\infty} S^{3+3\alpha})$ in their description is the unit map from the sphere spectrum to the spectrum $\mathbf{KQ}$.

In light of the discussion above involving \cite[Theorem 4.4.5]{AsokFaselKO}, the stabilization map then induces a commutative diagram of the form:
\begin{equation}
 \label{eq:25}
 \xymatrix{ & \KM_5/24 \ar^f[d] \ar[r] & \bpi^{\aone}_4(S^{3+3\alpha}) \ar[r] \ar[d] & \GW_4^3 \ar@{=}[d] \ar[r] & 0 \\ 0
 \ar[r] &
 \KM_5/24 \ar[r] & \bpi^{\aone}_4( \Omega_{\PP^1}^\infty \Sigma_{\PP^1}^{\infty} S^{3+3\alpha} ) \ar[r] & \GW_4^3
 \ar[r] & 0; }
\end{equation}
some comments are in order about this morphism of exact sequences. We have suppressed the tensoring with $\Z[1/p]$ on the two terms on the left in the bottom row since $p$ will, momentarily, be taken to be different from $2$ or $3$; our goal is to analyze the left-hand vertical map and, to this end, we use the following result about endomorphisms of unramified Milnor K-theory sheaves.

\begin{lem}
 \label{lem:selfmapquotientmilnorK}
 Let $n \ge 0$ and let $m$ be a nonnegative integer. Then
 \[ \Hom(\KM_n/m, \KM_n/m) = \ZZ/(m) \]
 generated by the identity map.
\end{lem}

\begin{proof}
The case $n=0$ is trivial. We observe that if $n\ge 1$, then by Lemma~\ref{lem:homfromkmwiscontraction}, $\Hom(\KM_n/m, \KM_n/m)$ is a subgroup of $(\KM_n/m)_{-n} = \KM_0/m = \ZZ/(m)$. But the $n$-fold contraction
 functor sends the identity map in $\Hom(\KM_n/m, \KM_n/m)$ to a generator of
 $\Hom(\KM_0/m, \KM_0/m) = \ZZ/(m)$.
\end{proof}

For an integer $n \geq 2$, set
\[
\nu_n := \Sigma^{(n-2)+(n-2)\alpha} \nu: S^{((n+1)+(n+2)\alpha)} \to S^{n+n\alpha}.
\]
We will abuse notation and write $\nu$ for the $\pone$-stable homotopy class corresponding to $\nu$. With this notation, we now establish Theorem~\ref{thmintro:maincomputation} from the introduction.

\begin{thm}
\label{thm:maincomputation}
Let $k$ be a field having characteristic different from $2$ or $3$.
\begin{enumerate}[noitemsep,topsep=1pt]
\item There is a short exact sequence of the form:
\begin{equation*}
\xymatrix{ 0 \ar[r] & \KM_5/24 \ar[r]^-{{\nu_3}_*} & \bpi^{\aone}_4(S^{3+3\alpha}) \ar[r] & \GW_4^3 \ar[r] & 0.}
\end{equation*}
\item For any integer $n \geq 4$, the morphism ${\nu_n}_*: \K^{MW}_{n+2} \to \bpi_{n+1}^{\aone}(S^{n+n\alpha})$ factors through an injection $0 \to \K^{\mathrm{M}}_{n+2}/24 \to \bpi_{n+1}^{\aone}(S^{n+n\alpha})$.
\end{enumerate}
\end{thm}

\begin{proof}
We would like to deduce that Diagram \eqref{eq:25} is essentially an isomorphism of short exact sequences. First of all, we show $f$ is an isomorphism. By Lemma~\ref{lem:selfmapquotientmilnorK}, it suffices to check this after $5$-fold contraction. The
diagram becomes
\begin{equation}
 \label{eq:25-5}
 \xymatrix{ & \Z/(24) \ar^{f_{-5}}[d] \ar[r] & \bpi^{\aone}_{4+5\alpha}(S^{3+3\alpha}) \ar[r] \ar[d] & 0 \ar[d] \ar[r] & 0 \\ 0
 \ar[r] &
 \Z/(24) \ar[r] & \bpi^{\aone}_{4+5\alpha}( \Omega_{\PP^1}^\infty \Sigma_{\PP^1}^{\infty} S^{3+3\alpha} ) \ar[r] &0
 \ar[r] & 0. }
\end{equation}
Since $\Z/(24)$ is independent of the field, we may verify $f$ is an isomorphism after passage to an algebraically
closed field. In characteristic $0$, it follows from \cite[Remark 5.7]{RSO} or \cite[Corollary 5.2.7]{AWW} that the $\Z/(24)$ in both the source and the target of $f_{-5}$ is generated by the motivic Hopf map $\nu_3$. Since $24$ is divisible only by $2$ and $3$, it follows using \'etale realization and lifting from characteristic $0$ that $\nu$ generates the stable group. Thus, in either case, the map sends a generator to a generator and must be an isomorphism. Therefore, the induced map $\KM_5/24 \to \KM_5/24$ is an isomorphism. Since there is an epimorphism $\K^{\mathrm{M}}_5/24 \to \mathbf{S}_5$ factoring this isomorphism, we also conclude that $\mathbf{S}_5 \cong \KM_5/24$.

The second point is established in an entirely similar fashion and we leave the details to the reader. In brief, \cite[Theorem 5.5]{RSO} gives a computation of $\bpi_{n+1}^{\aone}(\Omega_{\pone}^{\infty}\Sigma_{\pone}^{\infty}S^{n+n\alpha})$ for arbitrary integers $n$. Thus, one simply repeats the arguments above and again appeals to Lemma~\ref{lem:selfmapquotientmilnorK} to conclude.
\end{proof}


\subsubsection*{Suslin's conjecture in degrees $3$ and $5$}
We now analyze Suslin's conjecture, using the ideas developed above. In Remark~\ref{rem:degree3} we discuss Suslin's conjecture in degree $3$ in our context. The next result establish Theorem~\ref{thmintro:suslinsconjecture} from the introduction.

\begin{thm}
\label{thm:suslinsconjecturedegree5}
For any infinite field $k$ having characteristic unequal to $2$ or $3$ and any essentially smooth local $k$-algebra $A$, the Suslin--Hurewicz map $K^Q_5(A) \to K^M_5(A)$ has image precisely $4! K^M_5(A)$.
\end{thm}

\begin{proof}
In the proof of Theorem \ref{thm:maincomputation} we showed that the canonical epimorphism $\K^{\mathrm{M}}_5/24 \to \mathbf{S}_5$ is an isomorphism of strictly $\aone$-invariant sheaves. The result then follows immediately from Theorem \ref{thm:suslinmorphismscoincide} by evaluation at stalks (recall that $\aone$-homotopy sheaves have the property that the Zariski sheafification is already a Nisnevich sheaf, and so one may evaluate at Zariski stalks).
\end{proof}

\begin{rem}
\label{rem:degree3}
Our techniques also show that the Milnor conjecture on quadratic forms implies Suslin's conjecture in degree $3$, though we leave the details to the interested reader. Granted Theorem~\ref{thm:suslinmorphismscoincide}, Suslin's conjecture is essentially contained in F. Morel's identification of $\bpi_2^{\aone}(BGL_2) = \K^{MW}_2$ \cite[Theorem 7.20]{MField}. In more detail, Morel showed that the Milnor--Witt K-theory sheaves $\K^{MW}_n$ may be realized as the fiber product sheaves $\K^{\mathrm{M}}_n \times_{\K^{\mathrm{M}}_n/2} \mathbf{I}^n$. Granted the Milnor conjecture on quadratic forms, there are associated short exact sequences $0 \to \mathbf{I}^{n+1} \to \K^{MW}_n \to \K^{\mathrm{M}}_n \to 0$ and $0 \to 2\K^{\mathrm{M}}_n \to \K^{MW}_n \to \mathbf{I}^n \to 0$. The long exact sequence in $\aone$-homotopy sheaves associated with the $\aone$-fiber sequence ${\mathbb A}^3 \setminus 0 \to BGL_2 \to BGL_3$ takes the form
\[
\bpi_3^{\aone}(BGL_3) \longrightarrow \K^{MW}_3 \longrightarrow \K^{MW}_2 \longrightarrow \KM_2 \longrightarrow 0.
\]
By appealing to \'etale and real realization along the lines of the proof of Lemma~\ref{lem:scunit}, one may check if $k$ is a field having characteristic unequal to $2$, the epimorphism $\K^{MW}_2 \to \KM_2$ is the canonical epimorphism with kernel $\mathbf{I}^3$ and the induced map $\K^{MW}_3 \to \mathbf{I}^3$ is the canonical epimorphism with kernel $2\KM_3$. It follows that the image of $\bpi_3^{\aone}(BGL_3)$ in $\K^{MW}_3$, i.e., the image of $\psi_3$, is precisely $2\KM_3$, i.e., Suslin's conjecture holds in degree $3$.
\end{rem}

{\begin{footnotesize}
\raggedright
\bibliographystyle{alpha}
\bibliography{SuslinDegree5}
\end{footnotesize}
}
\Addresses
\end{document}